\documentclass[english, 12pt,leqno]{amsart}

\usepackage{a4,amsbsy,amsfonts,amsgen,
amsmath,
amsopn,amssymb,amsthm,amstext,enumerate,enumitem,exscale,
stmaryrd,epsfig,a4,graphicx,mathrsfs,relsize,url}

\usepackage{color}

\newtheorem{thm}{Theorem}

\newtheorem{lem}[thm]{Lemma}

\newtheorem{prop}[thm]{Proposition}
\newtheorem{reduc}[thm]{Reduction}
\theoremstyle{definition}

\newtheorem{exe}[thm]{Example}

\newtheorem{rem}[thm]{Remark}

\newtheorem{remi}[thm]{Reminder}

\newcommand{\R}{\mathbf R}
\newcommand{\C}{\mathbf C}
\newcommand{\Aut}{\rm Aut}
\newcommand{\U}{\mathbf U}
\newcommand{\M}{\rm M_2(\C)}
\newcommand{\id}{\rm id}
\newcommand{\tr}{\rm trace}
\newcommand{\SU}{\rm SU}
\newcommand{\SO}{\rm SO}
\newcommand{\GL}{\rm GL}
\newcommand{\Ad}{\rm Ad}
\newcommand{\diag}{\rm diag}
\newcommand{\supp}{\rm supp}
\newcommand{\EXP}{\rm EXP}

\title[On groups of smooth maps into a simple compact Lie group]
{On groups of smooth maps into a simple compact Lie group, revisited}

\date{9 January 2023}

\author{Pierre de la Harpe}

\address{Pierre de la Harpe,
Section de math\'ematiques,
Universit\'e de Gen\`eve,
\newline
Uni Dufour,
24 rue du G\'en\'eral Dufour,
Case postale 64,
CH--1211 Gen\`eve 4.
}

\email{Pierre.delaHarpe@unige.ch}



\begin{document}

\begin{abstract}
Let $X$ be a closed smooth manifold,
$G$ be a simple connected compact real Lie group,
$M (G)$ be the group of all smooth maps from $X$ to $G$,
and $M_0 (G)$ be its connected component
for the $\mathcal C^\infty$-compact open topology.

It is shown that maximal normal subgroups of $M_0 (G)$
are precisely the inverse images of the centre $Z(G)$ of $G$
by the evaluation homomorphisms
$M_0 (G) \to G, \hskip.1cm \gamma \mapsto \gamma (a)$, for $a \in X$.
This in turn is a consequence of a result
on the group $\mathcal C^\infty_{n, G}$ of germs at the origin $O$ of $\R^n$
of smooth maps $\R^n \to G$:
this group has a unique maximal normal subgroup,
which is the inverse image of $Z(G)$ by the evaluation homomorphism
$\mathcal C^\infty_{n, G} \to G, \hskip.1cm \underline \gamma \mapsto \underline \gamma (O)$.
\par

This article provides corrections for part of an earlier article \cite{Harp--88}.
%
%
\end{abstract}

\maketitle

\section*{Introduction}

An earlier article \cite{Harp--88} was found to contain several mistakes.
A list of local corrections would have been confusing,
and we rather repeat with complete proofs
Propositions I and II below (II and III in \cite{Harp--88}), about maximal subgroups.
Theorem I of \cite{Harp--88}, about automorphisms,
is replaced here by a conjectural statement only.

\vskip.2cm

Let $X$ be a closed smooth manifold
and let $G$ be a simple connected compact real Lie group,
with Lie algebra $\mathfrak g$.
Let $M (G)$ be the group of all smooth maps from $X$ to $G$,
and let $M_0 (G)$ be its connected component
for the $\mathcal C^\infty$-compact open topology
(the topology of uniform convergence of the maps and all their partial derivatives of all orders).
Then $M_0 (G)$ is an important example of a well behaved
infinite dimensional Lie group; see \cite{Miln--84, PrSe--86, Neeb--06, KhWe--09}.
However, $M_0 (G)$ is viewed here as an abstract group.
\par

Proposition I provides the classification
of all maximal normal subgroups of $M_0 (G)$.
For $a \in X$, let $\varepsilon_a \, \colon M_0 (G) \to G$
denote the evaluation map $\gamma \mapsto \gamma (a)$,
and let
$$
N_a M_0 (G) = \varepsilon_a^{-1} (Z(G))
$$
be the inverse image by $\varepsilon_a$
of the centre $Z(G)$ of $G$.
Since $G / Z(G)$ is simple,
$N_a M_0 (G)$ is a maximal normal subgroup of $M_0 (G)$.
Conversely:

\vskip.2cm
\noindent
\textbf{Proposition~I.}
\emph{Any proper normal subgroup of $M_0 (G)$
is contained in $N_a M_0 (G)$ for some $a \in X$.}
\vskip.2cm

One may think of Proposition~I as a global result,
of which the proof uses Proposition~II
which is a related local result.
Denote by $n$ the dimension of $X$,
by $\mathcal C^\infty_{n, G}$
the set of germs at the origin $O$ of $\R^n$
of smooth maps from $\R^n$ to $G$,
and by $\varepsilon \, \colon \mathcal C^\infty_{n, G} \to G$
the evaluation map $\underline \gamma \mapsto \underline \gamma (O)$.
Let
$$
N_O \mathcal C^\infty_{n, G} = \varepsilon^{-1} ( Z(G))
$$
be the inverse image of $Z(G)$
by $\varepsilon$.
Then $\mathcal C^\infty_{n, G}$ has a unique maximal normal subgroup;
in other words:

\vskip.2cm
\noindent
\textbf{Proposition~II.}
\emph{Any proper normal subgroup of $\mathcal C^\infty_{n, G}$
is contained in $N_O \mathcal C^\infty_{n, G}$.}
\vskip.2cm

Everything here works equally well for maps and germs
which are of class $\mathcal C^k$ for some $k \ge 0$.
The proof of Proposition~II works also in the real analytic setting.
Though a result analogous to Proposition~I
for real analytic maps looks plausible,
it is not covered by our proof, which uses partitions of unity.
Also, I guess that Proposition~II holds for a simple connected Lie group
which is not necessarily compact.
\par

Let $M (\Aut (G))$ be the group of smooth maps from $X$
to the group $\Aut (G)$ of automorphisms of~$G$
(recall that any automorphism in $\Aut (G)$ is automatically continuous \cite{Cart--30},
indeed analytic).
Let $\mathcal D (X)$ be the group of smooth diffeomorphisms of $X$.
Consider the natural action of $\mathcal D (X)$ on $M (\Aut (G))$, defined by
$$
\varphi (\beta) \, = \, \beta \circ \varphi^{-1}
\hskip.5cm \text{for} \hskip.2cm
\varphi \in \mathcal D (X)
\hskip.2cm \text{and} \hskip.2cm
\beta \in M (\Aut (G)),
$$
and the associated semi-direct product
$$
M (\Aut (G)) \rtimes \mathcal D (X) ,
\hskip.5cm \text{with multiplication} \hskip.2cm
(\alpha, \varphi) (\beta, \psi) \, = \, (\alpha \varphi(\beta), \varphi \psi) .
$$
This acts on $M_0 (G)$ by automorphisms:
$$
(\alpha, \varphi) (\gamma) (x) \, = \, \alpha(x) \big( \gamma \left( \varphi^{-1}(x) \right) \big)
$$
for $\alpha \in M (\Aut (G))$, $\varphi \in \mathcal D (X)$,
$\gamma \in M_0 (G)$, and $x \in X$.
We believe that any automorphism of the abstract group $M_0 (G)$ is of this form,
hence in particular that any automorphism of $M_0 (G)$ is continuous
for the $\mathcal C^\infty$-compact open topology.
In other terms, we can formulate the following statement,
which is the alleged Theorem I in \cite{Harp--88}:

\vskip.2cm
\noindent
\textbf{Conjectural Theorem.}
\emph{With the notation above, the group of all abstract group automorphisms of $M_0 (G)$
coincides with the group $M (\Aut (G)) \rtimes \mathcal D (X)$.}

\vskip.2cm

In Proposition 3.4.2 of \cite{PrSe--86},
Pressley and Segal claim that the group $M (\Aut (G)) \rtimes \mathcal D (X)$
is the group of all \emph{bicontinuous} automorphisms
of the topological group $M_0 (G)$.

\vskip.2cm

The article \cite{Harp--88} has clearly not been much read.
However, in November 2009, Michael Murray and Daniel Stevenson
pointed out a serious flaw in the so-called proof of Lemma~14 --- and
consequently also in the proof of Proposition~I.
(A problematic step is the sentence
``By Lemma~12 again one has $N_\mathcal W \subset \ker (\pi)$'';
this Lemma~14 of \cite{Harp--88} has been replaced by Lemma~\ref{Zeproblem} below.)
At the time, I was not smart enough (or stubborn enough) to write a correction.
I discussed the matter with Georges Skandalis,
who convinced me that repairing Lemma 14
(Lemma~\ref{Zeproblem} below) was feasible,
and I made a good resolution to work on this as soon as possible,
but this was delayed by several years.
\par

As I finally came back to the problem in 2022,
I found other gaps and defects in proofs of other lemmas.
(In particular, it is not appropriate to define as there a topology on a space of smooth germs.)
Short of an erratum restoring all claims and in particular Theorem I of \cite{Harp--88},
I decided to write up a complete proof of Propositions I and II.
Moreover, the so-called Theorem~I of \cite{Harp--88} is now degraded
to a ``Conjectural Theorem'' about the group of automorphims of the abstract group of $M_0 (G)$,
discussed in the introduction.
\par

The following concordance table indicates how lemmas below
correspond to lemmas in \cite{Harp--88}:
$$
\begin{aligned}
&
\begin{array}{c|cccccc}
\text{[Harp--88]}
& Lem~1 & Lem~2 & - & Lem~3 & Lem~4 & Lem~5
\\
\hline
Below
& Lem~1 & Lem~2 & Remark~3 & Lem~4 & Lem~5 & Lem~6
\\
\end{array}
\\
&
\begin{array}{c|cccccc}
\text{[Harp--88]}
& Prop~6 & - & - & Lem~7 & Lem~8 & Lem~9
\\
\hline
Below
& Prop~7 & Remind~8 & Reduct~9 & Lem~10 & Lem~11 & Lem~12
\\
\end{array}
\\
&
\begin{array}{c|cccccc}
\text{[Harp--88]}
& - & - & Lem~10 & Lem~11 & Lem~13 & Lem~12
\\
\hline
Below
& Remind~13 & Ex~14 & Lem~15 & Lem~16 & Lem~17 & Lem~18 
\\
\end{array}
\\
&
\begin{array}{c|cccc}
\text{[Harp--88]}
& --- & Lem~14 & Lem~15~\&~16
\\
\hline
Below
& Lem~19~\&~20~\&~21 & Lem~22 & -- 
\\
\end{array}
\end{aligned}
$$
(Lemmas 15 and 16 in \cite{Harp--88} are not correct
and should now be ignored).

\section{Proof of Proposition~II when $G = \SU (2)$}
\label{sectionpourIISU(2)}

Note that the result of Section~\ref{sectionpourIISU(2)}
for this particular case $G = \SU (2)$
will be used in the proof of the result for the general case,
see the proof of Lemma~\ref{normalcontaincsts} in Section~\ref{sectionpourIIGnal}.
\par

We adopt the following notation.
The multiplicative group of complex numbers of modulus $1$ is denoted by $\U$.
The algebra of linear endomorphisms of $\C^2$
is the algebra $\M$ of 2-by-2 matrices with complex coefficients;
$x^*$ is the conjugate transpose of a matrix $x \in \M$,
and $\id_2$ is the unit matrix.
The space of orthogonal projections from $\C^2$ onto lines is the projective line
$$
{\rm P}^1_\C \, = \, \{ p \in \M \mid p^* = p = p^2 \hskip.2cm \text{and} \hskip.2cm \tr (p) = 1 \} .
$$
In this section, until Proposition~\ref{propIISU(2)} included,
we use $G$ for the simple connected compact Lie group
$$
\begin{aligned}
\SU (2)
\, &= \,
\{ g \in \M \mid g^*g = \id_2 \hskip.2cm \text{and} \hskip.2cm \det g = 1 \} 
\\
\, &= \,
\left\{ \begin{pmatrix} \phantom{-}\rho & \sigma \\ -\overline{\sigma} & \overline{\rho} \end{pmatrix}
\in \M
\hskip.1cm \bigg\vert \hskip.1cm
\rho, \sigma \in \C , \hskip.2cm \vert \rho \vert^2 + \vert \sigma \vert^2 = 1 \right\} .
\end{aligned}
$$
An element $g \in G$ is \textbf{regular} if $g \notin \{\id_2, -\id_2\}$;
we denote by $G_{\rm reg}$ the open subset of $G$ of regular elements.
\par

Any $g \in G_{\rm reg}$ has two eigenvalues $z_g, \overline {z_g} \in \U$;
the notation is such that ${\Im (z_g) > 0}$.
Set
$$
A := \mathopen] 0, \pi \mathclose[ .
$$
Let $s_g \in A$ be the number such that $z_g = \exp (i s_g)$.
Let $p_g \in {\rm P}^1_\C$ be the projection of $\C^2$ onto the eigenspace $\ker (z_g \id_2 - g)$
and $p'_G \in {\rm P}^1_\C$ the projection of $\C^2$ onto $\ker (\overline {z_g}\id_2 - g)$.
We have
$$
g \, = \, \exp (i s_g) p_g + \exp (-i s_g) p'_g
\hskip.5cm \text{for all} \hskip.2cm
g \in G_{\rm reg} .
$$
\par

The group $G$ acts by conjugation on both $G_{\rm reg}$ and ${\rm P}^1_\C$,
trivially on~$A$, and by the product action on ${\rm P}^1_\C \times A$.

\begin{lem}
\label{SU(2)reg=PxA}
The map
\hskip.2cm
$\psi_A \, \colon \left\{ \begin{aligned}
G_{\rm reg} &\to {\rm P}^1_\C \times A \\ g \hskip.2cm &\mapsto (p_g, s_g)
\end{aligned} \right.$
\hskip.2cm
is a smooth diffeomorphism,
and it is equivariant for the action of $G$.
\end{lem}

\begin{proof}
Let $h \in G_{\rm reg}$.
Let $D_r$ be a closed disc of centre $\exp (is_h)$ and some radius~$r > 0$,
contained in the half-plane $\{ z \in \C \mid \Im z > 0 \}$,
and let $\Gamma_r$ be the circle $\partial D_r$, with the positive orientation.
Since eigenvalues depend continuously on matrices,
there exists a neighbourhood $\mathcal V$ of $h$ in $G_{\rm reg}$
such that $\exp (is_g)$ is in the interior of $D_r$ for all $g \in \mathcal V$.
For $g \in \mathcal V$, we have the explicit formula
$$
p_g = \frac{1}{2 i \pi} \int_{\Gamma_r} (z-g)^{-1} dz
$$
by holomorphic functional calculus.
It follows that $p_g$ depends smoothly on~$g$.
We have also
$$
s_g \, = \, \arccos \Big( \frac{1}{2} \tr (g) \Big) ,
\leqno{(*)}
$$
hence $s_g$ depends smoothly on~$g$.
\par

For the continuity of eigenvalues and for holomorphic functional calculus,
see for example
\cite[Chapter VII, \S~3, no.\ 19, Corollary 21]{DuSc--58},
or \cite[Chap.\ 1, \S~4, no~11, Proposition 16]{BoTS},
or
\cite[5.2.3 \& 5.5]{Serr--10},
or \cite{Texi--18}.
\par

The equivariance of $\psi_A$ is straightforward.
\end{proof}

The complex projective line can also be described as
$$
\begin{aligned}
{\rm P}^1_\C \, = \, &\left\{ \begin{pmatrix} a & b \\ \overline b & 1-a \end{pmatrix}
\hskip.2cm \Big\vert \hskip.2cm
a \in \R, \hskip.1cm b \in \C, \hskip.1cm 0 \le a \le 1, \hskip.1cm a^2 + \vert b \vert^2 = a
\right\} 
\\
\, = \, &\left\{
\begin{pmatrix}
\frac{1}{2} \left( 1 + \sqrt{1 - 4 r^2}\right) & r e^{i \varphi}
\\ r e^{-i \varphi} & \frac{1}{2} \left( 1 - \sqrt{1 - 4 r^2}\right)
\end{pmatrix}
\hskip.1cm \Bigg\vert \hskip.2cm
0 \le r \le \frac{1}{2} , \hskip.2cm \varphi \in \mathopen[ 0, 2\pi \mathclose[
\right\} \phantom{.}
\\
\, \cup \, &\left\{
\begin{pmatrix}
\frac{1}{2} \left( 1 - \sqrt{1 - 4 r^2}\right) & r e^{i \varphi}
\\ r e^{-i \varphi} & \frac{1}{2} \left( 1 + \sqrt{1 - 4 r^2}\right)
\end{pmatrix}
\hskip.1cm \Bigg\vert \hskip.2cm
\frac{1}{2} \ge r \ge 0 , \hskip.2cm \varphi \in \mathopen[ 0, 2\pi \mathclose[
\right\} .
\end{aligned}
$$
The second parametrization makes it clear that ${\rm P}^1_\C$ is diffeomorphic to the $2$-sphere,
shown as two hemispheres glued along the equator.
We define
$$
\begin{aligned}
\mathcal U 
\, &= \,
{\rm P}^1_\C \smallsetminus
\left\{ \begin{pmatrix} 0 & 0 \\ 0 &1 \end{pmatrix} \right\}
\\
\, &= \,
\left\{
\begin{pmatrix} a & b \\ \overline b & 1-a \end{pmatrix}
\hskip.2cm \Big\vert \hskip.2cm
a \in \R, \hskip.1cm b \in \C, \hskip.1cm 0 < a \le 1, \hskip.1cm a^2 + \vert b \vert^2 = a
\right\} ,
\end{aligned}
$$
which is an open neighbourhood of $\begin{pmatrix} 1 & 0 \\ 0 & 0 \end{pmatrix}$
in~${\rm P}^1_\C$.

\begin{lem}
\label{Uconjpolenord}
There exists a smooth map
\hskip.2cm
$ \left\{ \begin{aligned}
\mathcal U &\to G \\ u &\mapsto g_u
\end{aligned} \right.$
\hskip.2cm
such that $g_u \begin{pmatrix} 1 & 0 \\ 0 & 0 \end{pmatrix} g_u^{-1} = u$
for all $u \in \mathcal U$.
\end{lem}

\begin{proof}
Given a rank one projection
$u = \begin{pmatrix} a & b \\ \overline b & 1-a \end{pmatrix} \in \mathcal U$,
set $\rho = \sqrt{a}$, $\sigma = -b/\rho$, and define
$g_u = \begin{pmatrix} \phantom{-}\rho & \sigma \\ -\overline {\sigma} & \rho \end{pmatrix}
\in \SU (2)$.
Then
$g_u \begin{pmatrix} 1 & 0 \\ 0 & 0 \end{pmatrix} g_u^{-1} = u$,
by a simple computation.
\end{proof}

\begin{rem}
\label{obsanalytic}
The diffeomorphism of Lemma \ref{SU(2)reg=PxA}
and the map of Lemma~\ref{Uconjpolenord}
are not only smooth, they are indeed real analytic.
Similarly, the map of Lemma~\ref{Greg=G/TxA} below is real analytic.
\end{rem}

Let $\mathcal O$ be a neighbourhood of the origin in $\R^n$
and let $\gamma \, \colon \mathcal O \to G$
be a smooth map with values in $G_{\rm reg}$.
Denote by $\underline \gamma \in \mathcal C^\infty_{n, G}$ the germ defined by $\gamma$.
Then $s_{\gamma (x)}$ depends smoothly on $x$, by Lemma~\ref{SU(2)reg=PxA},
and defines at the origin a germ $\underline s_{\underline \gamma}$
of $A$-valued smooth map.
Similarly $p_{\gamma (x)}$ defines at the origin a germ $\underline p_{\underline \gamma}$
of ${\rm P}^1_\C$-valued smooth map.
We write simply $\underline s$ and $\underline p$
when the reference to $\underline \gamma$ is clear.

\begin{lem}
\label{lemconjgermsSU(2)}
Let $\underline \gamma$ and $\underline \delta$
be two germs in $\mathcal C^\infty_{n, G}$
such that $\underline \gamma (O), \underline \delta (O) \in G_{\rm reg}$.
Let $(\underline p, \underline s)$ and $(\underline q, \underline t)$
be the associated germs at $O$ of $({\rm P}^1_\C \times A)$-valued smooth maps.
\par

Then $\underline \gamma$ and $\underline \delta$
are conjugate in $\mathcal C^\infty_{n, G}$
if and only if $\underline s = \underline t$.
\end{lem}

\begin{proof}
Consider $\underline \zeta \in \mathcal C^\infty_{n, G}$.
We may choose representatives $\gamma, p, s, \delta, q, t, \zeta$ of
$\underline \gamma, \underline p, \underline s, \underline \delta, \underline q, \underline t,
\underline \zeta$
defined in a common neighbourhood $\mathcal O$ of the origin in $\R^n$,
and such that $\gamma(x), \delta(x) \in G_{\rm reg}$ for all $x \in \mathcal O$.
\par

Suppose first that
$\underline \zeta \hskip.1cm \underline \gamma \hskip.1cm \underline \zeta^{-1}
= \underline \delta$.
Upon replacing $\mathcal O$ by a smaller neighbourhood,
we may assume that $\zeta(x) \gamma(x) \zeta(x)^{-1} = \delta(x)$
for all $x \in \mathcal O$,
hence that $s(x) = t(x)$ for all $x \in \mathcal O$.
It follows that $\underline s = \underline t$.
\par

Suppose now that $\underline s = \underline t$.
We may again assume that $s(x) = t(x)$ for all $x \in \mathcal O$.
\par

Suppose first that, moreover, the range of $p(O)$ and the range of $q(O)$
are not orthogonal in $\C^2$.
In appropriate orthogonal coordinates:
$$
p(O) \, = \, \begin{pmatrix} 1 & 0 \\ 0 & 0 \end{pmatrix}
\hskip.3cm \text{and} \hskip.3cm
q(O) \, = \, \begin{pmatrix} c & d \\ \overline d & 1-c \end{pmatrix}
\hskip.3cm \text{with} \hskip.3cm
c > 0 , \hskip.1cm d \in \C , \hskip.1cm c^2 + \vert d \vert^2 = c .
$$
Upon shrinking $\mathcal O$ once more if necessary,
we may assume that $p$ and $q$ are of the form
$$
\begin{aligned}
p(x) \, &= \, \begin{pmatrix} a(x) & b(x) \\ \overline {b(x)} & 1 - a(x) \end{pmatrix}
\hskip.2cm \text{where} \hskip.2cm a(x) > 0
\hskip.2cm \text{for all} \hskip.2cm x \in \mathcal O ,
\hskip.2cm \text{with} \hskip.2cm a(O) = 1 ,
\\
q(x) \, &= \, \begin{pmatrix} c(x) & d(x) \\ \overline {d(x)} & 1-c(x) \end{pmatrix}
\hskip.2cm \text{where} \hskip.2cm c(x) > 0
\hskip.2cm \text{for all} \hskip.2cm x \in \mathcal O ,
\hskip.2cm \text{with} \hskip.2cm c(O) = c .
\end{aligned}
$$
By Lemma~\ref{Uconjpolenord}, there exists a smooth map $\zeta \, \colon \mathcal O \to G$
such that $\zeta(x) p(x) \zeta(x)^{-1} = q(x)$ for all $x \in \mathcal O$.
Since $s = t$, it follows that $\zeta \gamma \zeta^{-1} = \delta$, hence
$\underline \zeta \hskip.1cm \underline \gamma \hskip.1cm \underline \zeta^{-1} = \underline \delta$.
\par

Suppose now that the range of $p(O)$ and the range of $q(O)$
are orthogonal in ${\rm P}^1_\C$.
Define a third germ $\underline \eta \in \mathcal C^\infty_{n, G}$
with eigenvalue germ equal to $\underline s$ (hence also equal to $\underline t$),
and with $\underline \eta (O)$ orthogonal
neither to $\underline \gamma (O)$ nor to $\underline \delta (O)$.
The previous argument shows that $\underline \gamma$ and $\underline \delta$
are both conjugate to $\underline \eta$,
hence are conjugate one to the other.
\end{proof}

\begin{lem}
\label{lembigformula}
Let $\mathcal O$ be a neighbourhood of the origin in $\R^n$
and $s, t$ two smooth maps $\mathcal O \to A$ such that
$\sin (t(x) / 2) < \sin (s(x))$ for all $x \in \mathcal O$.
Define smooth maps $\zeta, \eta, \delta \, \colon \mathcal O \to G$ by
$$
\begin{aligned}
\zeta (x) \, &= \,
\begin{pmatrix}
\Big( 1 - \frac{ \sin^2 (t(x) / 2) }{ \sin^2 (s(x)) } \Big)^{1/2} & \frac{ \sin (t(x) / 2) }{ \sin (s(x)) } 
\\
- \hskip.1cm \frac{ \sin (t(x) / 2) }{ \sin (s(x)) } & \Big( 1 - \frac{ \sin^2 (t(x) / 2) }{ \sin^2 (s(x)) } \Big)^{1/2}
\end{pmatrix} 
\in \SO (2) \subset G ,
\\
\eta (x) \, &= \,
\begin{pmatrix} \exp (is(x)) & 0 \\ 0 & \exp (-is(x)) \end{pmatrix} \zeta (x)
\begin{pmatrix} \exp (-is(x)) & 0 \\ 0 & \exp (is(x)) \end{pmatrix} \zeta (x)^{-1} ,
\\
\delta (x) \, &= \,
\begin{pmatrix} \exp (it(x)) & 0 \\ 0 & \exp (-it(x)) \end{pmatrix} ,
\end{aligned}
$$
for all $x \in \mathcal O$.
\par

Then $\eta (x) \in G_{\rm reg}$ for all $x \in \mathcal O$,
and the germs of $\eta$ and $\delta$ are conjugate in $\mathcal C^\infty_{n, G}$.
\end{lem}

\begin{proof}
Note that $s(x), t(x) \in A$ implies $\zeta(x), \delta(x) \in G_{\rm reg}$
for all $x \in \mathcal O$.
A straightforward computation shows that, for all $x \in \mathcal O$,
$$
\begin{aligned}
\tr (\eta (x)) \, &= \, 2 \hskip.1cm \frac{ \sin^2 (s(x)) - (1 - \cos (2s(x))) \sin^2 (t(x) / 2) }{ \sin^2 s(x) }
\\
&= \, 2 ( 1 - 2 \sin^2 ( t(x)/2 ) )\, 
\, = \, 2 \cos (t(x)) \, = \, \tr (\delta (x)) .
\end{aligned}
$$
A first consequence is that $\eta (x) \in G_{\rm reg}$,
because $2 \cos (t(x)) \ne \pm 2$.
(Remember Formula (*) in the proof of Lemma~\ref{SU(2)reg=PxA}.)
A second consequence is that the $A\text{-valued}$ germs at $O$
associated to $\eta$ and $\delta$ are equal.
It follows from Lemma~\ref{lemconjgermsSU(2)}
that the germs at $O$ of $\eta$ and $\delta$ are conjugate in $\mathcal C^\infty_{n, G}$.
\end{proof}

\begin{lem}
\label{NcontainsV}
Let $N$ be a normal subgroup of $\mathcal C^\infty_{n, G}$ containing a germ
$\underline \gamma$ such that $\underline \gamma (O ) \in G_{\rm reg}$.
\par

There exists a symmetric neighbourhood $\mathcal V$ of the identity in $G$
(depending on $\underline \gamma$) such that
every $\underline \delta \in \mathcal C^\infty_{n, G}$
with $\underline \delta (O) \in \mathcal V \cap G_{\rm reg}$
is in $N$.
\end{lem}

\noindent
(A neighbourhood $\mathcal V$ of $\id_2$ in $G$ is symmetric
if $g^{-1} \in \mathcal V$ for all $g \in \mathcal V$.)

\begin{proof}
Let $(\underline p, \underline s)$ be the $({\rm P}^1_\C \times A)$-valued germ at $O$
associated to $\underline \gamma$, as in Lemma~\ref{lemconjgermsSU(2)}.
Let $\gamma, p, s$ be representatives of $\underline \gamma, \underline p, \underline s$
defined in a neighbourhood $\mathcal O$ of the origin in $\R^n$.
By this Lemma~\ref{lemconjgermsSU(2)}, there is no loss of generality
if we assume that
$$
\gamma (x) \, = \, \begin{pmatrix} \exp (is(x)) & 0 \\ 0 & \exp (-is(x)) \end{pmatrix}
$$
and that $s(x)$ is bounded away from $\{0, \pi\}$ for all $x \in \mathcal O$,
more precisely that there exists $c > 0$
such that $c < s(x) < \pi - c$ for all $x \in \mathcal O$.
\par 

Let $\mathcal V^\bullet$ be the open subset of $G_{\rm reg}$ of elements $g$
such that $\sin (s_g / 2) < \sin (c)$,
where $s_g$ is defined by $\psi_A(g) = (p_g, s_g)$,
where $\psi_A$ is as in Lemma~\ref{SU(2)reg=PxA}.
Set $\mathcal V = \mathcal V^\bullet \cup \{ \id_2 \}$;
it is a symmetric open neighbourhood of $\id_2$ in $G$.
Let $\underline \delta \in \mathcal C^\infty_{n, G}$ 
be such that $\underline \delta (O) \in \mathcal V^\bullet$;
we have to show that $\underline \delta \in N$.
\par

Let $(\underline q, \underline t)$ be the $({\rm P}^1_\C \times A)$-valued germ
associated to $\underline \delta$.
Upon replacing $\mathcal O$ by a smaller neighbourhood of the origin in $\R^n$,
we can find representatives $\delta, q, t$ of $\underline \delta, \underline q, \underline t$
such that $\delta (x) \in \mathcal V^\bullet$, i.e., such that $0 < \sin (t (x) / 2) < \sin (x)$,
for all $x \in \mathcal O$.
By Lemma~\ref{lemconjgermsSU(2)} again,
it suffices to consider the element $\underline \delta$ defined by
$$
\delta (x) \, = \, \begin{pmatrix} \exp (it(x)) & 0 \\ 0 & \exp (-it(x)) \end{pmatrix}
$$
for all $x \in \mathcal O$.
\par

By Lemma~\ref{lembigformula},
there exists a smooth map $\zeta \, \colon \mathcal O \to G$ such that
$\underline \delta$ and 
$\underline \gamma\hskip.1cm \underline \zeta \hskip.1cm
\underline \gamma^{-1} \hskip.1cm \underline \zeta^{-1}$
are conjugate in $\mathcal C^\infty_{n, G}$,
and therefore auch that $\underline \delta$
is in the normal subgroup generated by $\underline \gamma$.
\end{proof}

\begin{prop}
\label{propIISU(2)}
Proposition~II of the Introduction holds for $G = \SU (2)$:
\par

any proper normal subgroup of $\mathcal C^\infty_{n, G}$
is contained in $N_O \mathcal C^\infty_{n, G}$.
\end{prop}

\begin{proof}
Let $N$ be a normal subgroup of $\mathcal C^\infty_{n, G}$
which is not contained in $N_O \mathcal C^\infty_{n, G} = \varepsilon^{-1}(Z(G))$.
Then $\varepsilon (N) = G$,
because a normal subgroup of $G$ not contained in $Z(G)$ is $G$ itself
(see Reminder~\ref{RemiCartanvdW}).
In particular $N$ contains a germ $\underline \gamma$
such that $\underline \gamma (O)$ is regular.
Let $\underline \zeta \in \mathcal C^\infty_{n, G}$;
we have to show that $\underline \zeta \in N$.
\par

Let $\mathcal O$ be a neighbourhood of the origin in $\R^n$
and let $\gamma, \zeta \, \colon \mathcal O \to G$
be representatives of $\underline \gamma, \underline \zeta$.
By Lemma~\ref{NcontainsV}, there exists
a symmetric neighbourhood $\mathcal V$ of $\id_2$ in $G$
such that every $\underline \delta \in \mathcal C^\infty_{n, G}$
with $\underline \delta (O) \in \mathcal V \cap G_{\rm reg}$ is in $N$.
Since $G$ is connected, there exists an integer $k \ge 1$
and elements $g_1, g_2, \hdots g_k \in \mathcal V \cap G_{\rm reg}$
such that $\zeta(O) = g_1 g_2 \cdots g_k$.
For any $g \in G$, denote by $\eta_g \, \colon \mathcal O \to G$
the constant map of value $g$.
Set $\delta = \zeta (\eta_{g_2} \cdots \eta_{g_k})^{-1}$, so that
$\zeta = \delta \eta_{g_2} \cdots \eta_{g_k}$.
Then the $k$ matrices
$\delta(O) = g_1, \eta_{g_2}(O) = g_2, \hdots, \eta_{g_k}(O) = g_k$
are all in $\mathcal V \cap G_{\rm reg}$.
Lemma \ref{NcontainsV} implies that the germs 
$\underline \delta, \underline {\eta_{g_2}}, \hdots \underline {\eta_{g_k}}$
are all in $N$, so that their product $\underline \zeta$ is also in~$N$.
\end{proof}

\begin{remi}
\label{RemiCartanvdW}
Let $G$ be a connected compact Lie group $G$
which is simple as a Lie group, i.e., such that its Lie algebra is simple.
The following results are due to Cartan and van der Waerden \cite{Cart--30, vdWa--33}.
\par

(i)
Any non-trivial normal subgroup of $G$ is contained in the centre $Z(G)$ of $G$,
and the quotient $G / Z(G)$ is simple as an abstract group.

(ii)
Any abstract group homomorphism with bounded image from $G$ to a Lie group is continuous,
and therefore smooth.
In particular any abstract group endomorphism $\phi$ of $G$ is necessarily smooth. 
Since the Lie algebra of $G$ is simple, the derivative of $\phi$ 
is either an isomorphism, in which case $\phi$ is an automorphism,
or zero, in which case $\phi(g) = 1_G$ for all $g \in G$.
\par

For the historical context of these articles by Cartan and van der Waerden,
see \cite[Chap.\ VI, \S~6]{Bore--01}.
For hints of proofs,
see \cite[Chapitre III, \S~4 exercices 8 \& 9, and \S~9 exercice 25]{BoLG2-3}.
Recall that any continuous homomorphism between Lie groups is analytic,
and in particular is smooth;
see for example \cite[Chapitre III, \S~8, no.\ 1]{BoLG2-3}.
\end{remi}

We end this section by a digression.
Germs which are regular can be diagonalized by Lemma~\ref{lemconjgermsSU(2)},
but the regularity condition cannot be removed.
\par

The following example illustrates this;
it is a minor adaptation of one in \cite[Chap.~1, \S~3]{Rell--69}.
Set $n = 1$. Define first a map $\theta$ from $\R$
to the space of $2$-by-$2$ matrices by
$$
\theta(x) \, = \,
\left\{ \begin{aligned}
\exp (-x^{-2}) \begin{pmatrix} \cos (2/x) & \phantom{-} \sin (2/x) \\ \sin (2/x) & -\cos (2/x) \end{pmatrix}
\hskip.2cm &\text{if} \hskip.2cm x \ne 0 ,
\\
\begin{pmatrix} 0 & 0 \\ 0 & 0 \end{pmatrix}
\hskip2cm &\text{if} \hskip.2cm x = 0 .
\end{aligned} \right.
$$
Define then a map $\gamma \, \colon \R \to \SU (2)$
by $\gamma (x) = \exp (i\theta(x))$ for all $x \in \R$.
Outside the origin, $\gamma(x)$ has eigenvectors
$\begin{pmatrix} \cos (1/x) \\ \sin (1/x) \end{pmatrix}$ and
$\begin{pmatrix} \phantom{-} \sin (1/x) \\ - \cos (1/x) \end{pmatrix}$
with eigenvalues
$\exp (i \exp (-x^{-2}) )$ and $\exp (-i \exp (-x^{-2}) )$.
But there is no germ $\underline \zeta$ such that
$\underline \zeta \hskip.1cm \underline \gamma \hskip.1cm \underline \zeta^{-1}$
is diagonal.
\par

Even though it is not relevant here, we mention that
one may diagonalize $\underline \gamma$ with a Borel map \cite{Azof--74}.

\section{Proof of Proposition~II in the general case}
\label{sectionpourIIGnal}

Let $G$ be simple connected compact Lie group,
$\mathfrak g$ its Lie algebra,
$\exp \, \colon \mathfrak g \to G$ the exponential map,
and $\Ad \, \colon G \to \GL(\mathfrak g \otimes_\R \C)$
the adjoint representation of~$G$.

\begin{reduc}
\label{Reduction}
It is sufficient to prove Proposition~II when $G$ is simply connected.
\end{reduc}

\begin{proof}
Indeed, suppose that Proposition~II holds for the universal cover $\widetilde G$ of $G$
(which is still compact by Weyl's theorem \cite[\S~1, no~4]{BoLG9}).
The short exact sequence
$$
\{ 1 \} \, \to \, \pi_1(G) \, \to \, \widetilde G \, \to \, G \, \to \, \{ 1 \}
$$
induces a sequence
$$
\{ 1 \} \, \to \, \pi_1(G) \, \to \, \mathcal C^\infty_{n, \widetilde G}
\, \overset{p}{\to} \, \mathcal C^\infty_{n, G} \, \to \, \{ 1 \}
$$
which is again exact
(in these sequences, elements of $\pi_1(G)$
are viewed as elements of the centre of $\widetilde G$,
and also as germs of constant maps
from $\R^n$ to the centre of $\widetilde G$).
We denote here by $\varepsilon_G \, \colon \mathcal C^\infty_{n, G} \to G$ and
$\varepsilon_{\widetilde G} \, \colon \mathcal C^\infty_{n, \widetilde G} \to \widetilde G$
the evaluation maps at the origin of $\R^n$.
\par

Let $N$ be a normal subgroup of $\mathcal C^\infty_{n, G}$
which is not contained in $\varepsilon_G^{-1}(Z(G))$.
As $\widetilde G / Z(\widetilde G) = G / Z(G)$,
the normal subgroup $\widetilde N := p^{-1}(N)$ of $C^\infty_{n, \widetilde G}$
is not contained in $\varepsilon_{\widetilde G}^{-1}(Z(\widetilde G))$.
If Proposition~II holds for $\widetilde G$,
then $\widetilde N = \mathcal C^\infty_{n, \widetilde G}$,
hence $N = \mathcal C^\infty_{n, G}$;
thus Proposition~II holds for $G$.
\end{proof}

For $g \in G$, denote by $Z_G(g)_0$ the connected component
of the centralizer $Z_G(g) = \{ h \in G \mid hg = gh \}$ of $g$.
The element $g$ is \textbf{regular} if $Z_G(g)_0$ is a maximal torus in $G$,
equivalently if there exists a unique maximal torus in $G$ containing~$g$.
When $g$ is regular and $h \in G$ is not, $\dim Z_G(g)_0 < \dim Z_G(h)_0$
and there are infinitely many maximal tori in $G$ containing $h$.
We denote by $G_{\rm reg}$ the set of regular elements in $G$.
Conjugates of regular elements are regular,
so that $G$ acts on $G_{\rm reg}$ by conjugation.
The subset $G_{\rm reg}$ is symmetric, open, and dense in $G$.
More on regular elements in \cite[Chap.\ VII, \S~4]{BoLG7-8}
and \cite[\S~5, no~1]{BoLG9}.
We 
\vskip.2cm

\begin{center}
\emph{choose a maximal torus $T$ in $G$}
\end{center}

\vskip.2cm
\noindent
and we denote by $\mathfrak t$ its Lie algebra.
We choose an alcove $A$ in $\mathfrak t$,
namely a connected component of the subset of $\mathfrak t$
consisting of those $\xi \in \mathfrak t$ with $\exp (\xi)$ regular.
The group $G$ acts as usual on $G/T$, trivially on $A$,
and by conjugation on the subset $G_{reg}$ of regular elements.
Because of Reduction~\ref{Reduction}:

\vskip.2cm

\begin{center}
\emph{from now on in this section, we assume that $G$ is simply connected.}
\end{center}

\vskip.2cm

Lemmas~\ref{Greg=G/TxA}
and \ref{lemconjgermsGnal}
can be seen as generalizations of
Lemmas~\ref{SU(2)reg=PxA}
and \ref{lemconjgermsSU(2)} respectively.

\begin{lem}
\label{Greg=G/TxA}
The map
\hskip.2cm
$\varphi_A \, \colon \left\{ \begin{aligned}
G/T \times A &\to \hskip.5cm G_{\rm reg} \\ (gT, s) \hskip.2cm &\mapsto g (\exp (s) ) g^{-1}
\end{aligned} \right.$
\hskip.2cm
is a smooth diffeomorphism,
and it is equivariant for the action of $G$.
\end{lem}

\begin{proof}
The map $\varphi_A$ is a diffeomorphism,
see \cite[Proposition~4b of Page~51]{BoLG9};
the group denoted there by $H_A$ is here reduced to $\{ 1 \}$,
because $G$ is simply connected,
see \cite[Remark~1 of Page 45]{BoLG9}.
The equivariance is clear.
\end{proof}

We denote by $\pi \, \colon G \to G/T$ the canonical projection.

\begin{lem}
\label{deloctriv}
There exist an integer $k \ge 1$ and
\begin{enumerate}[label=(\roman*)]
\item[]
open subsets $\mathcal U_1, \mathcal U_2, \hdots, \mathcal U_k$ in $G/T$,
\item[]
smooth maps $\mu_j \, \colon \mathcal U_j \to G$ for $j = 1, 2, \hdots, k$,
\end{enumerate}
such that
\begin{enumerate}[label=(\roman*)]
\item[]
$\pi \big( \mu_j (u) \big) = u$ for each $u \in \mathcal U_j$, for $j = 1, 2, \hdots, k$.
\end{enumerate}
\end{lem}

\begin{proof}
This is a way to say that $\pi \, \colon G \to G/T$ is the projection
of a locally trivial bundle with compact base $G/T$ and fibre $T$.
\end{proof}

\begin{lem}
\label{lemconjgermsGnal}
Let $\underline \gamma$ and $\underline \delta$
be two germs in $\mathcal C^\infty_{n, G}$
such that $\underline \gamma (O), \underline \delta (O) \in G_{\rm reg}$.
Let $(\underline p, \underline s)$ and $(\underline q, \underline t)$
be the associated germs at $O$ of $(G/T \times A)$-valued smooth maps
obtained by composition of $\underline \gamma$ and $\underline \delta$
with the map $\varphi_A^{-1}$, where $\varphi_A$ is as in Lemma~\ref{Greg=G/TxA}.
\par

Then $\underline \gamma$ and $\underline \delta$
are conjugate in $\mathcal C^\infty_{n, G}$
if and only if $\underline s = \underline t$.
\end{lem}

\begin{proof}
We prove the non-trivial implication only.
We assume that $\underline s = \underline t$
and we have to show that $\underline \gamma$ and $\underline \delta$ are conjugate.
We use the notation of Lemma~\ref{deloctriv}.
We distinguish two cases.
\par

In the first case, there exists $j \in \{1, \hdots, k\}$
such that $\underline p (O), \underline q (O)$ are in $\mathcal U_j$.
We may choose representatives
$\gamma, p, s, \delta, q, t$ of
$\underline \gamma, \underline p, \underline s, \underline \delta, \underline q, \underline t$
defined in a common neighbourhood $\mathcal O$ of the origin in $\R^n$,
such that $p(x), q(x) \in \mathcal U_j$
and $s(x) = t(x)$ for all $x \in \mathcal O$.
Define a smooth map $\zeta \, \colon \mathcal O \to G$ by
$\zeta(x) = \mu_j(q(x)) \big( \mu_j(p(x)) \big)^{-1}$;
then $\zeta(x) p(x) = q(x)$ for all $x \in \mathcal O$.
(Note that $\zeta(x)$ is an element of $G$ which acts on $G/T$,
and that $\zeta(x) p(x) = q(x)$ is an equality between elements of $G/T$.)
By Lemma~\ref{Greg=G/TxA} we have
$\zeta(x) \gamma(x) \zeta(x)^{-1} = \delta(x)$ for all $x \in \mathcal O$,
so that $\underline \gamma$ and $\underline \delta$ are conjugate.
\par

In the second case, there exist $j,j' \in \{1, \hdots, k\}$ with $j \ne j'$
such that $\underline p (O) \in \mathcal U_j$
and $\underline q (O) \in \mathcal U_{j'}$.
We may choose representatives
$\gamma, p, s, \delta, q, t$ of
$\underline \gamma, \underline p, \underline s,
\underline \delta, \underline q, \underline t$
defined in a common neighbourhood $\mathcal O$ of the origin in $\R^n$,
such that $p(x) \in \mathcal U_j$, $q(x) \in \mathcal U_{j'}$,
and $s(x) = t(x)$ for all $x \in \mathcal O$.
There exist an integer $m \ge 1$ and a sequence
$q_0 = \underline p (O), q_1, \hdots, q_{m-1}, q_m = \underline q (O)$
of elements of $G/T$ such that,
for each $i \in \{1, \hdots, m\}$,
there exists $j(i) \in \{1, \hdots, k\}$
with both $q_{i-1}$ and $q_i$ in $\mathcal U_{j(i)}$.
For $i \in \{1, \hdots, m-1 \}$, define a smooth map
$\gamma_i \, \colon \mathcal O \to G$ 
by $\gamma_i (x) = \varphi_A ( q_i, s(x))$ for all $x \in \mathcal O$;
write $\gamma_0$ for $\gamma$ and $\gamma_m$ for $\delta$.
By the argument for the first case,
there exists for $i \in \{1, \hdots, m \}$
a smooth map $\zeta_i \, \colon \mathcal O \to G$
such that $\zeta_i (x) q_{i-1} (x) = q_i (x)$, and therefore
$\zeta_i (x) \gamma_{i-1} (x) \zeta_i (x)^{-1} = \gamma_i (x)$,
for all $x \in \mathcal O$.
It follows that $\zeta_m \zeta_{m-1} \cdots \zeta_1$ conjugates
$\gamma_0 = \gamma$ to $\gamma_m = \delta$,
and therefore that $\underline \gamma$ is conjugate to~$\underline \delta$.
\end{proof}

As we have not been able to generalize Lemma~\ref{lembigformula},
we proceed in the next lemmas by reduction to the case of $\SU (2)$.
Before this, we recall the following facts
on the structure of simple connected compact Lie groups.

\begin{remi}
\label{RemiSsgpsSU(2)}
Let $X(T)$ denote the group of continuous homomorphisms
from the maximal torus $T$ to the group $\U$ of complex numbers of modulus $1$.
For $\alpha \in X(T)$, set
$$
\mathfrak g^\alpha_\C = \{ \xi \in \mathfrak g \otimes_\R \C \mid
\Ad (t) \xi = \alpha(t) \xi \hskip.2cm \text{for all} \hskip.2cm t \in T \} .
$$
The group $X(T)$ is written additively,
so that $\alpha = 0$ is the constant homomorphism $T \to \U$ of value $1$.
A root of $G$ with respect to $T$ is an element $\alpha \ne 0$ in $X(T)$
such that $\dim \mathfrak g^\alpha_\C > 0$.
If $\alpha$ is a root, it is known
that $\dim \mathfrak g^\alpha_\C = 1$ and that $-\alpha$ is also a root.
Let $R(G, T)$ denote the set of roots.
An element $t \in T$ is regular if and only if
$t$ is not contained in $\ker (\alpha)$ for all $\alpha \in R(G, T)$,
so that the alcove $A$ chosen above is a connected component of
$\mathfrak t \smallsetminus \bigcup_{\alpha \in R(G, T)}
\{ \xi \in \mathfrak t \mid \exp ( \xi ) \in \ker (\alpha) \}$.
Choose a basis
of the set of roots $R(G, T)$
and let $R_+(G, T)$ be the corresponding set of positive roots.

\vskip.2cm

\emph{Let $\alpha$ be a root.
There exists a connected closed subgroup $S_\alpha$ of $G$
and a surjective morphism of Lie groups
$$
\nu_\alpha \, \colon \SU (2) \to S_\alpha
$$
such that
\begin{enumerate}[label=(\roman*)]
\item\label{iDESU(2)radiciel}
$S_\alpha$ and the kernel of $\alpha \, \colon T \to \U$ commute,
\item\label{iiDESU(2)radiciel}
$\nu_\alpha \begin{pmatrix} a & 0 \\ 0 & \overline a \end{pmatrix} \in T$
and $\alpha
\left( \nu_\alpha \begin{pmatrix} a & 0 \\ 0 & \overline a \end{pmatrix} \right)
= a^2$
for all $a \in \U$,
\item\label{iiiDESU(2)radiciel}
$T = \big( S_\alpha \cap T \big) \cdot \big( \ker (\alpha) \big)$,
\item\label{ivDESU(2)radiciel}
$\nu_\alpha$ is injective, i.e., $\nu_\alpha \, \colon \SU (2) \to S_\alpha$ is an isomorphism.
\end{enumerate}
The group $S_\alpha$ is the Lie subgroup of $G$ of Lie algebra
$$
\mathfrak{su} (2)_\alpha = \mathfrak g \cap
\left( \mathfrak g^\alpha_\C \oplus \mathfrak g^{-\alpha}_\C
\oplus [ \mathfrak g^\alpha_\C , \mathfrak g^{-\alpha}_\C ] \right) ,
$$
which is a subalgebra of $\mathfrak g$ isomorphic to $\mathfrak{su} (2)$.
}

\vskip.2cm

For the existence of $\nu_\alpha$, and~\ref{iDESU(2)radiciel} and~\ref{iiDESU(2)radiciel},
see \cite[\S~4, no~5, Page 31]{BoLG9}.
\par
For~\ref{iiiDESU(2)radiciel}, consider $t \in T$.
Choose a square root $a$ of $\alpha (t)$.
Set $t_1 = \nu_\alpha \begin{pmatrix} a & 0 \\ 0 & \overline a \end{pmatrix}$
and $t_2 = t_1^{-1}t$. Then $t = t_1t_2$ with $t_1 \in S_\alpha \cap T$ and $t_2 \in \ker (\alpha)$.
\par
For~\ref{ivDESU(2)radiciel},
see the Remarque in \cite[\S~4, no~5, Page 32]{BoLG9},
and recall that $G$ is simply connected.
(For $G = \SO(2n+1)$, which is not simply connected,
and for $\alpha$ a short root,
the image of $\nu_\alpha$ is isomorphic to $\SO (3) = \SU (2) / \{ \pm \id_2 \}$.)

\vskip.2cm

Let $\psi_\alpha$ denote the automorphism of $G$ of conjugation by
$\nu_\alpha
\begin{pmatrix} \phantom{-}2^{-1/2} & 2^{-1/2} \\ - 2^{-1/2} & 2^{-1/2} \end{pmatrix}$.
Then

\emph{
\begin{enumerate}[label=(\roman*)]
\addtocounter{enumi}{4}
\item\label{vDESU(2)radiciel}
$\psi_\alpha \left( \nu_\alpha \begin{pmatrix} a & 0 \\ 0 & \overline a \end{pmatrix} \right)
= 
\nu_\alpha \begin{pmatrix}
\frac{a + \overline a}{2} & \frac{-a + \overline a}{2} \\
\frac{-a + \overline a}{2} & \frac{a + \overline a}{2}
\end{pmatrix}$
for all $\nu_\alpha \begin{pmatrix} a & 0 \\ 0 & \overline a \end{pmatrix}
\in {\rm Im}(\nu_\alpha) \cap T$,
\item\label{viDESU(2)radiciel}
$\psi_\alpha (t) = t$ for all $t \in \ker (\alpha)$.
\item\label{viiDESU(2)radiciel}
$t \psi_\alpha(t)^{-1} \in \nu_\alpha \left( \SU (2)_{\rm reg} \right)$
for all $t \in T \cap G_{\rm reg}$.
\end{enumerate}
}

\vskip.2cm

Claim~\ref{vDESU(2)radiciel} follows from the computation
$$
\begin{pmatrix} \phantom{-}2^{-1/2} & 2^{-1/2} \\ - 2^{-1/2} & 2^{-1/2} \end{pmatrix}
\begin{pmatrix} a & 0 \\ 0 & \overline{a} \end{pmatrix}
\begin{pmatrix} 2^{-1/2} & -2^{-1/2} \\ 2^{-1/2} & \phantom{-}2^{-1/2} \end{pmatrix}
=
\begin{pmatrix}
\frac{ a + \overline{a} }{2} & \frac{ -a + \overline{a} }{2}
\\
\frac{ -a + \overline{a} }{2} & \frac{ a + \overline{a} }{2} 
\end{pmatrix}
$$
and Claim~\ref{viDESU(2)radiciel} follows from~\ref{iDESU(2)radiciel}.
(Compare~\ref{vDESU(2)radiciel} and~\ref{viDESU(2)radiciel}
with \cite[\S~4, no~5, Page 35]{BoLG9},
where the relevant inner automorphism of $G$ is that defined by
$\nu_\alpha \begin{pmatrix} \phantom{-}0 & 1 \\ -1 & 0 \end{pmatrix}$.)
\par

Let us show~~\ref{viiDESU(2)radiciel}.
Let $t \in T$. There exist by \ref{iiiDESU(2)radiciel} commuting elements
$$
t_1 = \nu_\alpha \begin{pmatrix} a & 0 \\ 0 & \overline a \end{pmatrix}
\in {\rm Im}(\nu_\alpha) \cap T
\hskip.5cm \text{and} \hskip.5cm
t_2 \in \ker (\alpha)
\hskip.5cm \text{such that} \hskip.5cm
t = t_1 t_2 .
$$
Then
$$
\begin{aligned}
t \psi_\alpha (t)^{-1}
&\hskip.2cm = \,
t_1 t_2
\nu_\alpha
\begin{pmatrix} \phantom{-}2^{-1/2} & 2^{-1/2} \\ - 2^{-1/2} & 2^{-1/2} \end{pmatrix}
t_2^{-1} t_1^{-1 }
\nu_\alpha
\begin{pmatrix} 2^{-1/2} & -2^{-1/2} \\ 2^{-1/2} & \phantom{-}2^{-1/2} \end{pmatrix}
\\
\, &\overset{\text{by \ref{iDESU(2)radiciel}}}{=} \,
t_1
\nu_\alpha
\begin{pmatrix} \phantom{-}2^{-1/2} & 2^{-1/2} \\ - 2^{-1/2} & 2^{-1/2} \end{pmatrix}
t_1^{-1 }
\nu_\alpha
\begin{pmatrix} 2^{-1/2} & -2^{-1/2} \\ 2^{-1/2} & \phantom{-}2^{-1/2} \end{pmatrix}
\\
&\hskip.2cm = \,
\nu_\alpha \Bigg( \frac{1}{2}
\begin{pmatrix} a & 0 \\ 0 & \overline a \end{pmatrix}
\begin{pmatrix} \phantom{-} 1 & 1 \\ -1 & 1 \end{pmatrix}
\begin{pmatrix} \overline a & 0 \\ 0 & a \end{pmatrix}
\begin{pmatrix} 1 & -1 \\ 1 & \phantom{-}1 \end{pmatrix}
\Bigg)
\\
&\hskip.2cm = \,
\nu_\alpha
\begin{pmatrix} \frac{1 + a^2}{2} & \frac{-1 + a^2}{2} \\
\frac{1 - \overline a^2}{2} & \frac{1 + \overline a^2}{2} \end{pmatrix} .
\end{aligned} .
$$
If $t \in T \cap G_{\rm reg}$, then $t \notin \ker (\alpha)$, hence $t_1 \notin \ker (\alpha)$,
hence $a^2 \ne 1$ by \ref{iiDESU(2)radiciel},
and therefore
$\begin{pmatrix} \frac{1 + a^2}{2} & \frac{-1 + a^2}{2} \\
\frac{1 - \overline a^2}{2} & \frac{1 + \overline a^2}{2} \end{pmatrix}
\in \SU (2)_{\rm reg}$.

\vskip.2cm

Finally, we have the following, used below in the proof of Lemma~\ref{perfection}.

\emph{
\begin{enumerate}[label=(\roman*)]
\addtocounter{enumi}{7}
\item\label{viiiDESU(2)radiciel}
There exist a neighbourhood $\mathcal V$ of $1_G$ in $G$
and for each $\alpha \in R_+(G, T)$
a smooth map $\lambda_\alpha \, \colon \mathcal V \to S_\alpha$ such that
$\prod_{\alpha \in R_+(G,T)} \nu_\alpha( \lambda_\alpha (g) ) = g$.
\end{enumerate}
}

\vskip.2cm

\noindent
\emph{Proof of~\ref{viiiDESU(2)radiciel}.}
We write $R_+$ for $R_+(G, T)$.
For $\alpha \in R_+$, recall that $\mathfrak{su} (2)_\alpha$ is the Lie algebra of $S_\alpha$. 
Consider the surjective linear map
$\bigoplus_{\alpha \in R_+} \mathfrak{su} (2)_\alpha \to \mathfrak g$
which applies $(\xi_\alpha)_{\alpha \in R_+}$ onto $\sum_{\alpha \in R_+} \xi_\alpha$.
Choose for each $\alpha \in R_+$ a linear map
$\kappa_\alpha \, \colon \mathfrak g \to \mathfrak{su} (2)_\alpha$
such that $\sum_{\alpha \in R_+} \kappa_\alpha (\xi) = \xi$
for all $\xi \in \mathfrak g$.
\par

Let $\mathcal V'$ be a neighbourhood of $1_G$ in $G$
and $\mathcal W$ be a neighbourhood of the origin $O$ in $\mathfrak g$
such that the exponential map is a diffeomorphism $\exp \, \colon \mathcal W \to \mathcal V'$;
denote by $\log$ the inverse diffeomorphism.
For each $\alpha \in R_+$, define a smooth map
$\lambda'_\alpha \, \colon \mathcal V' \to S_\alpha$
by $\lambda'_\alpha (g) = \exp \Big( \kappa_\alpha \big( \log (g) \big) \Big)$.
Define next a smooth map $\lambda' \, \colon \mathcal V' \to G$
by $\lambda'(g) = \prod_{\alpha \in R_+} \nu_\alpha( \lambda'_\alpha (g) )$;
the terms of this product are arranged according to some chosen order,
and the same order is used for the next product $\prod_{\alpha \in R_+} (\cdots)$ below.
Observe that $\lambda'(1_G) = 1_G$ and that the derivative of $\lambda'$ at $1_G$
is the identity map $\mathfrak g \to \mathfrak g$.
Upon replacing $\mathcal V'$ by a smaller neighbourhood of $1_G$,
one may therefore assume that $\lambda'$ is a diffeomorphism
from $\mathcal V'$ to some neighbourhood $\mathcal V$ of $1_G$ in $G$.
We denote by $\nu$ the inverse diffeomorphism.
For each $\alpha \in R_+$, define a smooth map
$\lambda_\alpha \, \colon \mathcal V \to S_\alpha$
by $\lambda_\alpha (g) = \lambda'_\alpha ( \nu(g) )$.
Then $\prod_{\lambda \in R_+} \nu_\alpha( \lambda_\alpha (g) )
= \prod_{\lambda \in R_+} \nu_\alpha( \lambda'_\alpha ( \nu(g) ) )
= \lambda' ( \nu(g) ) = g$
for all $g \in G$, and this concludes the proof.
\hfill $\square$
\end{remi}

%
%

\begin{exe}
Let $G = \SU (n)$, for $n \ge 3$,
and let $T$ be the torus of diagonal matrices in~$G$.
The open dense subset $T \cap G_{\rm reg}$ of $T$ consists of
matrices $\diag (z_1, z_2, \hdots, z_n)$
such that $z_1, z_2, \hdots, z_n$ are all distinct and $z_1 z_2 \cdots z_n = 1$.
Let $\alpha$ be the homomorphism $G \to \U$
mapping $\diag (z_1, z_2, \hdots, z_n)$ to $z_1z_2^{-1}$.
Then $\alpha$ is a root of $G$ with respect to $T$,
the image of $\nu_\alpha$ is
$\begin{pmatrix}
\SU(2) & 0 \\ 0 & \id_{n-2}
\end{pmatrix}$
and the kernel of $\alpha$ is
$\begin{pmatrix}
\pm \id_2 & 0 \\ 0 & \SU(n-2)
\end{pmatrix}$.
\end{exe}

\begin{lem}
\label{normalcontaincsts}
Let $N$ be a normal subgroup of $\mathcal C^\infty_{n, G}$
which contains a germ $\underline \gamma$
such that $\underline \gamma (O) \notin Z(G)$.
\par

Then $N$ contains anll constant germs.
\end{lem}

\begin{proof}
\emph{Step one.} We can assume that $\underline \gamma (O) \in G_{{\rm reg}}$.
\par

Indeed, set $g = \underline \gamma (O)$.
Denote by $k$ the dimension of $G$;
define
$$
\mathcal M (g) = \left\{ h \in G \mid
h = \prod_{j=1}^k a_j b_j g b_j^{-1} g^{-1} a_j^{-1}
\hskip.2cm \text{for some} \hskip.2cm
a_1, b_1, \hdots, a_k, b_k \in G \right\} .
$$
Since $g$ is not central, $\mathcal M (g)$ is a neighbourhood of~$1_G$ in $G$
(this is the essential point in van der Waerden's proof
of the simplicity of the abstract group $G / Z(G)$ \cite{vdWa--33}).
Choose $a_1, b_1, \hdots, a_k, b_k \in G$
such that $g' \Doteq \prod_{j=1}^k a_j b_j g b_j^{-1} g^{-1} a_j^{-1} \in G_{{\rm reg}}$.
Recall that $\eta_a$ denote the constant map $\R^n \to G$ of value $a$.
Define
$$
\underline \gamma' = \prod_{j=1}^k
\underline \eta_{a_j} \hskip.1cm
\underline \eta_{b_j} \hskip.1cm
\underline \gamma \hskip.1cm
\underline \eta_{b_j}^{-1} \hskip.1cm
\underline \gamma^{-1} \hskip.1cm
\underline \eta_{a_j}^{-1} .
$$
Then $\underline \gamma'$ is in $N$ and $g' = \underline \gamma'(O) \in G_{{\rm reg}}$.
So that we can indeed assume from the start that $\underline \gamma (O) \in G_{{\rm reg}}$.

\vskip.2cm

\emph{Step two.}
We can assume that $\gamma (x) \in T \cap G_{\rm reg}$ for all $x \in \mathcal O$.
\par

Denote by $(\underline p, \underline s)$
the germ at $O$ of $(G/T \times A)$-valued smooth maps
associated to $\underline \gamma$ as in Lemma~\ref{lemconjgermsGnal}.
Let $\gamma \, \colon \mathcal O \to G_{\rm reg}$ and $s \, \colon \mathcal O \to A$
be representatives of $\underline \gamma$ and $\underline s$.
By Lemma~\ref{lemconjgermsGnal},
the germ at $O$ of the map
$\gamma' \, \colon \mathcal O \to G_{\rm reg}, \hskip.1cm x \mapsto \exp ( s(x))$
is conjugate to $\underline \gamma$, hence is in~$N$.
Therefore, we may assume from now on
that $\underline \gamma$ is itself the germ of the map
$\gamma \, \colon \mathcal O \to G_{\rm reg}, \hskip.1cm x \mapsto \exp ( s(x))$,
in particular that $\gamma (x) \in T \cap G_{\rm reg}$ for all $x \in \mathcal O$.

\vskip.2cm

\emph{Step three.}
The group $N$ contains the germ of one constant map with value not in $Z(G)$.
\par

Let $\alpha$ be a root of $G$ with respect to $T$.
Let $\nu_\alpha \, \colon \SU (2) \to G$ and $\psi_\alpha \in \Aut (G)$
be as in Reminder~\ref{RemiSsgpsSU(2)}.
Let $I_\alpha \, \colon \mathcal C^\infty_{n, \SU (2)} \to \mathcal C^\infty_{n, G}$
be the map which applies
the germ at $O$ of a smooth map $\mathcal O \to \SU (2)$
to the germ at $O$ of the composition $\mathcal O \to \SU (2) \overset{\nu_\alpha}{\to} G$;
then $I_\alpha$ is an injective homomorphism of groups.
Define the map $\delta : \mathcal O \to G$ by
$\delta (x) = \gamma (x) \psi_\alpha ( \gamma (x))^{-1}$ for all $x \in \mathcal O$
and let $\underline \delta$ denote its germ at $O \in \mathcal O$.
Since $\delta$ is the commutator of $\gamma$
with something,
$\underline \delta$ is in $N$.
It follows from \ref{viiDESU(2)radiciel} in Reminder~\ref{RemiSsgpsSU(2)}
that $\delta(x) \in \nu_\alpha \left( \SU (2)_{\rm reg} \right)$
for all $x \in \mathcal O$
and that $\underline \delta$ is in the image of $I_\alpha$.
Therefore $I_\alpha^{-1} (N)$ is a normal subgroup of $\mathcal C^\infty_{n, \SU (2)}$
which contains the germ $I_\alpha^{-1} (\underline \delta)$,
and $I_\alpha^{-1} (\underline \delta) (O) \in \SU(2)_{\rm reg}$.
It follows from Proposition~\ref{propIISU(2)} that $N$ contains
all germs at $O$ of smooth maps $\mathcal O \to \nu_\alpha ( \SU(2) )$.
\par

In particular, $N$ contains the germ of a constant map 
$\mathcal O \to G$ with value in $\nu_\alpha ( \SU (2)_{\rm reg} )$,
in particular with value in $G$ and not in $Z(G)$.

\vskip.2cm

\emph{Final step.}
Since any normal subgroup of $G$ not inside $Z(G)$ is $G$ itself
(see Reminder~\ref{RemiCartanvdW}),
$N$ contains the germs of all constant maps $\mathcal O \to G$.
\end{proof}

The following lemma is the last step in the proof of Proposition~II.

\begin{lem}
\label{coupgraceII}
Let $N$ be a normal subgroup of $\mathcal C^\infty_{n, G}$
which contains a germ $\underline \gamma$
such that $\underline \gamma (O) \notin Z(G)$.
\par

Then $N = \mathcal C^\infty_{n, G}$.
\end{lem}

\begin{proof}
\emph{Step one.}
Let $W$ be the Weyl group of $G$ with respect to $T$
(it is the quotient of the normalizer ${\rm Norm}_G (T)$ by $T$),
let $c \in W$ be a Coxeter element,
and let $m \in {\rm Norm}_G (T)$ be an element of class $c$.
Recall that $W$ acts on $T$ and $\mathfrak t$,
in particular $c$ acts on $T$ by $t \mapsto mtm^{-1}$
and on $\mathfrak t$ by $\xi \mapsto \Ad (m) \xi$.
Consider the smooth map
$$
f \, \colon T \to T , \hskip.5cm
t \mapsto c(t) t^{-1} = mtm^{-1}t^{-1} .
$$
The derivative of $f$ at the identity is
$$
L(f) \, \colon \mathfrak t \to \mathfrak t , \hskip.5cm
\xi \mapsto c( \xi )- \xi = \Ad (m) \xi - \xi .
$$
Now $L(f)$ is a linear automorphism, because Coxeter elements
do not have $1$ as an eigenvalue (see \cite[Page 33]{BoLG9},
referring itself to \cite[Chap.\ V, \S~6, no~2]{BoLG4-6}).
By the implicit function theorem,
there exist a symmetric neighbourhood~$\mathcal V$ of the identity in $T$
and a smooth map $\chi \, \colon \mathcal V \to T$ such that
$$
t = m \chi(t) m^{-1} \chi(t)^{-1}
\hskip.5cm \text{for all} \hskip.2cm
t \in \mathcal V .
\leqno{(\sharp)}
$$
Observe that $\bigcup_{g \in G} g \mathcal V g^{-1}$
is a neighbourhood of $1_G$ in $G$.

\vskip.2cm

\emph{Step two}.
Let $\underline \delta \in \mathcal C^\infty_{n, G}$ be such that
$\underline \delta (O) \in 
\big( \bigcup_{g \in G} g \mathcal V g^{-1} \big) \cap G_{\rm reg}$;
then $\underline \delta \in N$.
\par

Let $\mathcal O$ be a neighbourhood of the origin in $\R^n$
and let $\delta \, \colon \mathcal O \to
\big( \bigcup_{g \in G} g \mathcal V g^{-1} \big) \cap G_{\rm reg}$
be a smooth map of germ $\underline \delta$.
By Lemma~\ref{lemconjgermsGnal}, there exists a smooth map
$\zeta \, \colon \mathcal O \to \mathcal V \cap G_{\rm reg}$
of which the germ $\underline \zeta$
is conjugate to $\underline \delta$ in $\mathcal C^\infty_{n, G}$,
so that it suffices to show that $\underline \zeta \in N$.
Let $\eta_m$ be the constant map $\mathcal O \to G$ of value $m$;
by ($\sharp$), we have
$$
\zeta (x) = \eta_m (x) \chi(\zeta (x)) \eta_m (x)^{-1} \chi(\zeta (x))^{-1}
\hskip.5cm \text{for all} \hskip.2cm
x \in \mathcal \mathcal O .
$$
Since the germ of $\eta_m$ is in $N$ by Lemma~\ref{normalcontaincsts},
if follows that $\underline \zeta \in N$.

\vskip.2cm

\emph{Final step.}
Let now $\underline \delta$ be arbitrary in $\mathcal C^\infty_{n, G}$,
and let $\delta \, \colon \mathcal O \to G$ be a representative of $\underline \delta$.
Since $\big( \bigcup_{g \in G} g \mathcal V g^{-1} \big) \cap G_{\rm reg}$
is a nonempty symmetric open subset of $G$
and since $G$ is connected,
there exist $g_1, \hdots, g_k \in
\big( \bigcup_{g \in G} g \mathcal V g^{-1} \big) \cap G_{\rm reg}$
such that $\underline \delta (O) = g_1 \cdots g_k$.
For $j \in \{1, \hdots, k\}$, let
$\eta_j : \mathcal O \to \bigcup_{g \in G} g \mathcal V g^{-1}$
be the constant map of value $g_j$.
Set $\varphi = \delta (\eta_2 \cdots \eta_k)^{-1}$,
so that $\delta = \varphi \eta_2 \cdots \eta_k$.
Then $\varphi (O) = g_1, \eta_2 (O) = g_2, \hdots, \eta_k (O) = g_k$
are all in $( \bigcup_{g \in G} g \mathcal V g^{-1} \big) \cap G_{\rm reg}$,
hence $\underline \varphi, \underline \eta_2, \hdots, \underline \eta_k$
are all in $N$ by Step two.
It follows that
$\underline \delta = \underline \varphi \underline \eta_2 \cdots \underline \eta_k$
is in $N$.
\end{proof}

\section{Proof of Proposition~I}
\label{sectionPropI}

Let $X$ be a closed smooth manifold, $n$ its dimension,
$G$ a simple connected compact real Lie group,
$1_G$ the identity in $G$,
and $M_0 (G)$ the connected component
of the group $M (G)$ of all smooth maps from $X$ to $G$,
as in Proposition~I.
\par

In the following lemma, $G$ is identified with the group of constant maps in $M_0 (G)$.
Given subsets $A, B$ in $M_0 (G)$, we write $(A, B)$ for the subgroup of $M_0 (G)$
generated by the commutators $\gamma \delta \gamma^{-1} \delta^{-1}$
with $\gamma \in A$ and $\delta \in B$.

\begin{lem}
\label{perfection}
We have $M_0 (G) = (G, M_0 (G)$. 
In particular, the group $M_0 (G)$ is perfect.
\end{lem}

\emph{Note.}
The group $M (G)$ need not be perfect.
For example, when $X$ is the circle $S^1$ and $G$ is not simply connected,
there is a natural epimorphism from $M (G)$
to the fundamental group of $G$, which is abelian and not trivial,
so that $M (G)$ is not perfect.
\par

There
are cases where $M_0 (G) = M (G)$.
When $G$ is simply connected, this holds for $X$ the circle (obvious)
and for $X$ the $2$-sphere (because the second homotopy group of $G$
is trivial, by a theorem of Weyl).
When $G = \SU (2)$, this holds whenever
the third cohomotopy set of $X$ has one element,
in particular when $X$ is the circle or a closed surface.

\begin{proof}
We follow essentially the proof of Proposition~3.4.1 in \cite{PrSe--86}.
We consider first the case of $G = \SU (2)$.

\vskip.2cm

\emph{Step one.}
Let $T_1$ be the torus of diagonal matrices in $G$.
Let $\gamma \in M_0 (T_1)$. Then $\gamma \in (G, M_0 (G))$.
\par

Indeed, let $\mathcal U_1 = \left\{ \begin{pmatrix} e^{is} & 0 \\ 0 & e^{-is} \end{pmatrix} \in T_1
\hskip.2cm \Big\vert \hskip.2cm -\pi < s < \pi \right\}$.
Let
$$
\mathcal V_1 = \{ \gamma \in M (T_1) \mid \gamma (x) \in \mathcal U_1
\hskip.2cm \text{for all} \hskip.2cm x \in X \} ,
$$
which is a neighbourhood of the identity in $M_0(T_1)$.
Since there is a smooth retraction by deformation of $\mathcal U_1$ to $\{1_{T_1}\}$ in $T_1$,
any smooth map $\gamma \in \mathcal V_1$ is in $M_0 (T_1)$.
Suppose first that $\gamma \in \mathcal V_1$.
Then $\gamma$ is of the form
$x \mapsto \begin{pmatrix} e^{is(x)} & 0 \\ 0 & e^{-is(x)} \end{pmatrix}$
for some smooth map $s \, \colon X \to \mathopen] -\pi, \pi \mathclose[$.
Since
\small
$$
\begin{pmatrix} e^{is(x)} & 0 \\ 0 & e^{-is(x)} \end{pmatrix}
=
\begin{pmatrix} \phantom{-}0 & 1 \\ -1 & 0 \end{pmatrix}
\begin{pmatrix} e^{-is(x)/2} & 0 \\ 0 & e^{is(x)/2} \end{pmatrix}
\begin{pmatrix} 0 & -1 \\ 1 & \phantom{-}0 \end{pmatrix}
\begin{pmatrix} e^{is(x)/2} & 0 \\ 0 & e^{-is(x)/2} \end{pmatrix}
$$
\normalsize
for all $x \in X$, we have $\gamma \in (G, M_0 (G))$.
Let now $\gamma$ be arbitrary in $M_0 (T_1)$.
Since $M_0 (T_1)$ is connected,
there exist $\gamma_1, \hdots, \gamma_k \in \mathcal V_1$
such that $\gamma$ is the product $\gamma_1 \cdots \gamma_k$.
As each $\gamma_j$ is contained in $(G, M_0 (G))$ by the previous argument,
$\gamma \in (G, M_0 (G))$.

\vskip.2cm

\emph{Step two.}
Let
$$
\begin{aligned}
T_2 &= \left\{ \begin{pmatrix} \cos (t) & i \sin (t) \\ i \sin (t) & \cos (t) \end{pmatrix}
\hskip.2cm \Big\vert \hskip.2cm
t \in \mathopen[ 0, 2\pi \mathclose[ \right\} ,
\\
T_3 &= \left\{
\begin{pmatrix} \phantom{-} \cos (u) & \sin (u) \\ - \sin (u) & \cos (u) \end{pmatrix}
\hskip.2cm \Big\vert \hskip.2cm
u \in \mathopen[ 0, 2\pi \mathclose[
\right\} .
\end{aligned}
$$
Similarly, if $\gamma$ is in $M_0 (T_2)$ or in $M_0 (T_3)$,
then $\gamma \in (G, M_0(G))$.
\par

This follows from the Step one because maximal tori in $\SU (2)$ are conjugate.

\vskip.2cm

\emph{Step three.}
There is a neighbourhood of the identity $\mathcal W$ in $M_0 (G)$
such that every $\gamma \in \mathcal W$ is in $(G, M_0 (G))$.
\par

For $j \in \{1, 2, 3\}$, let $\mathfrak t_j$ denote the Lie algebra of $T_j$.
Consider the multiplication map
$$
\mu \, \colon
T_1 \times T_2 \times T_3 \to G, \hskip.2cm
(g_1, g_2, g_3) \mapsto g_1 g_2 g_3 .
$$
Its derivative at the identity
$$
\mathfrak t_1 \times \mathfrak t_2 \times \mathfrak t_3 \to \mathfrak g, \hskip.2cm
(\xi_1, \xi_2, \xi_3) \mapsto \xi_1 + \xi_2 + \xi_3
$$
is a linear isomorphism.
By the implicit function theorem,
there exist a neighbourhood $\mathcal V_T$ of the identity in $T_1 \times T_2 \times T_3$
and a neighbourhood $\mathcal V_G$ of the identity in $G$
such that the restriction of $\mu$ is a homeomorphism $\mathcal V_T \to \mathcal V_G$.
Upon restricting these neighbourhoods, we can assume moreover that
there exist smooth retractions by deformation
of $\mathcal V_T$ to $\{1_T\}$ in $T_1 \times T_2 \times T_3$
and of $\mathcal V_G$ to $\{1_G\}$ in $G$.
Define $\mathcal W = \{ \gamma \in M (G) \mid
\gamma(x) \in \mathcal V_G \hskip.2cm \text{for all} \hskip.2cm x \in X \}$;
any smooth map $\gamma \in \mathcal W$ is in $M_0 (G)$,
and $\mathcal W$ is a neighbourhood of the identity in $M_0 (G)$.
Then every $\gamma \in \mathcal W$ is a product $\gamma_1 \gamma_2 \gamma_3$
with $\gamma_j \in M_0(T_j)$ for $j \in \{1, 2, 3 \}$, hence $\gamma \in (G, M_0 (G))$
by the first two steps.

\vskip.2cm

\emph{Step four.}
Every $\gamma \in M_0 (G)$ is in $(G, M_0 (G))$.
\par

Since $M_0 (G)$ is connected, this follows from Step three by the usual argument.
(Compare with the end of the proof of Step one.)

\vskip.2cm

The part in which $G$ stands for $\SU (2)$ is now finished.
From now $G$ denotes again a general simple connected compact real Lie group,
as in the statement of Lemma~\ref{perfection}.

\vskip.2cm

\emph{Step five.}
Choose a maximal torus $T$ of $G$
and a basis
of the set $R = R(G, T)$ of roots of $G$ with respect to $T$.
Let $R_+$ be the corresponding set of positive roots.
Let $\mathcal V$ be a neighbourhood of $1_G$ in $G$
and, for each $\alpha \in R_+$,
let $\lambda_\alpha \, \colon \mathcal V \to S_\alpha$
be a smooth map such that
$\prod_{\alpha \in R_+} \nu_\alpha \big( \lambda_{\alpha} (g) \big) = g$ for all $g \in \mathcal V$.
(See~\ref{viiiDESU(2)radiciel} in Reminder~\ref{RemiSsgpsSU(2)};
recall that $\nu_\alpha$ is the inclusion homomorphism of $S_\alpha \approx \SU (2)$ into~$G$.)
\par

Let $\gamma \in M_0 (G)$ be such that $\gamma(X) \subset \mathcal V$.
For each $\alpha \in R_+$, the map $\lambda_\alpha \circ \gamma \, \colon X \to S_\alpha$
is in the commutator subgroup $( S_\alpha , M_0 (S_\alpha) )$,
by Step four. It follows that
$\gamma = \prod_{\alpha \in R_+} \nu_\alpha \big( \mu_\alpha \circ \gamma \big)$
is in the commutator subgroup $( G, M_0 (G) )$.
\par

Let now $\gamma$ be arbitrary in $M_0 (G)$. Since $M_0 (G)$ is connected,
$\gamma$ is a product of elements with images in $\mathcal V$,
hence $\gamma \in ( G, M_0 (G) )$.
\end{proof}

The support of an element $\gamma$ of $M (G)$ is the closure $\supp (\gamma)$
of the subset of~$X$ of points $x$ such that $\gamma (x) \ne 1_G$.

\begin{lem}
\label{partunity}
Let $(\mathcal U_i)_{1 \le i \le k}$ be an open covering of $X$
and let $\gamma \in M_0 (G)$.
\par

There exist a finite sequence of smooth maps $(\gamma_j)_{ 1 \le j \le \ell }$ in $M_0 (G)$
and a sequence of indices $(i(j))_{ 1 \le j \le \ell }$ in $\{ 1, \hdots, k \}$
with $\supp (\gamma_j) \subset \mathcal U_{i(j)}$ for $j \in \{1, \hdots, \ell\}$
such that $\gamma = \gamma_1 \cdots \gamma_\ell$.
\end{lem}

\begin{proof}
Let $\mathcal C^\infty(X, \mathfrak g)$ denote the space
of all smooth maps from $X$ to the Lie algebra $\mathfrak g$ of $G$.
Then $M_0 (G)$ is a Fr\'echet Lie group with Lie algebra $\mathcal C^\infty(X, \mathfrak g)$;
there is an exponential map
$$
\EXP \, \colon \mathcal C^\infty(X, \mathfrak g) \to M_0 (G)
$$
given by
$$
(\EXP (\zeta)) (x) = \exp ( \zeta(x) )
\hskip.5cm \text{for all} \hskip.2cm
\zeta \in \mathcal C^\infty(X, \mathfrak g)
\hskip.2cm \text{and} \hskip.2cm
x \in X ,
$$
where $\exp \, \colon \mathfrak g \to G$
is the exponential map of $G$,
and this $\EXP$ is a local diffeomorphism.
Since $M_0 (G)$ is connected,
there exist $\zeta_1, \hdots \zeta_m \in \mathcal C^\infty(X, \mathfrak g)$ such that
$$
\gamma = \EXP (\zeta_1) \cdots \EXP (\zeta_m) .
$$
Let $(\lambda_i)_{ 1 \le i \le k }$ be a smooth partition of unity
subordinated to $(\mathcal U_i)_{ 1 \le i \le k }$.
Since the $\lambda_i \zeta_1$~'s commute with each other, we have
$$
\EXP (\zeta_1) = \prod_{i = 1}^k \EXP (\lambda_i \zeta_1) 
$$
and $\supp (\lambda_i \zeta_1) \subset \mathcal U_i$ for $i \in \{1, \hdots, k\}$.
The same holds for each of $\EXP (\zeta_2)$, $\hdots$, $\EXP(\zeta_m)$,
and the lemma follows, with $N = k m$.
\end{proof}

For $a \in X$, recall that $\varepsilon_a$ denotes the evaluation map
$M_0 (G) \to G, \gamma \mapsto \gamma (a)$.
The kernel $\ker (\varepsilon_a)$ is a closed normal subgroup of $M_0 (G)$.
For an open subset $\mathcal U$ of $X$, define
$$
N_\mathcal U = \left\{ \gamma \in M_0 (G)
\mid
\supp (\gamma) \subset \mathcal U \right\} ,
$$
which is a normal subgroup of $M_0 (G)$.
For $a \in X$,
we write $N_{X \smallsetminus a}$ for $N_{X \smallsetminus \{a\}}$.

\begin{lem}
\label{connectivity+}
Let $\mathcal V$ be an open neighbourhood of $1_G$ in $G$
and $\mathcal W$ an open neighbourhood of $O$ in $\mathfrak g$
such that $\exp \, \colon \mathcal W \to \mathcal V$ is a diffeomorphism.
Let $a \in X$.
\begin{enumerate}[label=(\roman*)]
\item\label{iDEconnectivity+}
The normal closed subgroup $\ker (\varepsilon_a)$ of $M_0 (G)$ is connected.
\end{enumerate}
Let $\gamma \in N_{X \smallsetminus a}$
\begin{enumerate}[label=(\roman*)]
\addtocounter{enumi}{1}
\item\label{iiDEconnectivity+}
There exist an open neighbourhood $\mathcal U$ of $a$ in $X$
and a homotopy $(\gamma_t)_{0, \le t \le 1}$ in $M_0 (G)$
from $\gamma = \gamma_0$
to the constant map $\gamma_1 \, \colon X \to \{1_G\}$
such that $\gamma_t(x) = 1_G$
for all $t \in \mathopen[ 0 , 1 \mathclose]$ and all $x \in \mathcal U$.
\item\label{iiiDEconnectivity+}
Moreover, there exists
a sequence $\xi_1, \hdots, \xi_k \in \mathcal C^\infty(X, \mathfrak g)$
with the following properties
$$
\begin{aligned}
&\xi_i (x) \in \mathcal W
\hskip.5cm \text{for all} \hskip.2cm
i \in \{1, \hdots, k\}
\hskip.2cm \text{and} \hskip.2cm
x \in X ,
\\
&\xi_i(x) = 0 \hskip.5cm \text{for all} \hskip.2cm
i \in \{1, \hdots, k\}
\hskip.2cm \text{and} \hskip.2cm
x \in \mathcal U ,
\\
&\gamma (x) = \exp \big( \xi_1 (x) \big) \exp \big( \xi_2 (x) \big) \cdots \exp \big( \xi_k (x) \big)
\hskip.5cm \text{for all} \hskip.2cm
x \in X .
\end{aligned}
$$
\end{enumerate}
\end{lem}

\begin{proof}
\ref{iDEconnectivity+}
Let $\gamma \in M_0 (G)$.
\par

Since $M_0 (G)$ is by definition a connected group,
there exists a continuous map
$$
\Gamma_\gamma \, \colon X \times \mathopen[ 0, 1 \mathclose] \to G
$$
such that 
$$
\begin{aligned}
(11) \hskip.5cm
& \Gamma_\gamma ( \cdot , t ) \, \colon X \to G
\hskip.5cm \text{is smooth for all} \hskip.2cm
t \in \mathopen[ 0, 1 \mathclose] ,
\\
(12) \hskip.5cm
& \Gamma_\gamma ( x , 0 ) = \gamma (x)
\hskip.5cm \text{for all} \hskip.2cm
x \in X ,
\\
(13) \hskip.5cm
& \Gamma_\gamma ( x , 1 ) = \gamma(a)
\hskip.5cm \text{for all} \hskip.5cm
x \in X .
\end{aligned}
$$
When $\gamma \in \ker (\varepsilon_a)$,
note that $\Gamma_\gamma (\cdot, t)$
need not be in $\ker (\varepsilon_a)$ for $t \in \mathopen] 0, 1 \mathclose[$,
so that Claim~\ref{iDEconnectivity+} is not yet proved.
\par

Define
$$
\Delta'_\gamma \, \colon X \times \mathopen[ 0, 1 \mathclose] \to G
\hskip.5cm \text{by} \hskip.5cm
\Delta'_\gamma (x , t) =
\Gamma_\gamma ( x , t ) \Gamma_\gamma ( a , t )^{-1} \gamma(a)
$$
It follows from (11), (12), and (13), that
$$
\begin{aligned}
(21) \hskip.5cm
& \Delta'_\gamma ( \cdot , t ) \, \colon X \to G
\hskip.5cm \text{is smooth for all} \hskip.2cm
t \in \mathopen[ 0, 1 \mathclose] ,
\\
(22) \hskip.5cm
& \Delta'_\gamma ( x , 0 ) = \gamma (x) \gamma (a)^{-1} \gamma (a) = \gamma (x)
\hskip.5cm \text{for all} \hskip.2cm
x \in X ,
\\
(23) \hskip.5cm
& \Delta'_\gamma ( x , 1 ) = \gamma (a) \gamma (a)^{-1} \gamma (a) = \gamma (a)
\hskip.5cm \text{for all} \hskip.2cm
x \in X ,
\\
(24) \hskip.5cm
& \Delta'_\gamma ( a, t )
= \Gamma_\gamma (a,t) \Gamma_\gamma(a, t)^{-1} \gamma (a) = \gamma (a)
\hskip.5cm \text{for all} \hskip.2cm
t \in \mathopen[ 0, 1 \mathclose] .
\end{aligned}
$$
\par

Suppose now that $\gamma \in \ker (\varepsilon_a)$.
Then $(x \mapsto \Delta'_\gamma( x , t) )_{t \in \mathopen[ 0, 1 \mathclose]}$
is a continuous family of maps connecting $\gamma$
to the constant map $X \to \{ \gamma (a) \} = \{ 1_G \}$
in the group $M_0 (G)$, indeed in the group $\ker (\varepsilon_a)$ by (24).
This proves Claim~\ref{iDEconnectivity+}.

\vskip.2cm

\ref{iiDEconnectivity+}
Let $\mathcal U'$ be an open neighbourhood of $a$ in $X$
such that $\gamma (x) = \gamma (a)$ for all $x \in \mathcal U'$.
Upon replacing $\mathcal U'$ by a smaller neighbourhood,
we can assume that $\mathcal U'$ is diffeomorphic
to an open ball and that $\mathcal U'$ has spherical coordinates
$(r, \sigma) \in \mathopen[ 0, 1 \mathclose[ \times S^{n-1}$;
set $\mathcal U = \{ (r, \sigma) \in \mathcal U' \mid r < 1/3 \}$.
\par

Let $\lambda \, \colon \mathopen[ 0, 1 \mathclose[ \to \mathopen[ 0, 1 \mathclose[$
be a smooth function such that
$\lambda (t) = 0$ for $t \le 1/3$ and $\lambda (t) = 1$ for $t \ge 2/3$.
Define a smooth map $h \, \colon X \to X$ by
$h(x) = h( \lambda(r), \sigma)$ for all $x = (r, \sigma) \in \mathcal U'$
and $h(x) = x$ for all $x \in X \smallsetminus \mathcal U'$. 
Then
$$
\begin{aligned}
& (31) \hskip.5cm
h (x) = x
\hskip.5cm \text{for all} \hskip.2cm
x \in X \smallsetminus \mathcal U' ,
\\
& (32) \hskip.5cm
h (\mathcal U') \subset \mathcal U',
\\
& (33) \hskip.5cm
h(x) = a
\hskip.5cm \text{for all} \hskip.2cm
x \in \mathcal U .
\\
& (34) \hskip.5cm
\gamma \circ h = \gamma .
\end{aligned}
$$
[Check of (34):
if $x \in X \smallsetminus \mathcal U'$ we have $\gamma (h(x)) = \gamma(x)$
by (31)
and if $x \in \mathcal U'$ we have $\gamma (h(x)) = \gamma (a) = \gamma (x)$
by (33).]
\par

Let $\Delta'_\gamma$ be as in~\ref{iDEconnectivity+}. Define
$$
\Delta_\gamma \, \colon X \times \mathopen[ 0, 1 \mathclose] \to G
\hskip.5cm \text{by} \hskip.5cm
\Delta_\gamma (x , t) = \Delta'_\gamma( h(x) , t ) .
$$
We have
$$
\begin{aligned}
(41) \hskip.5cm
& \Delta_\gamma ( \cdot , t ) \, \colon X \to G
\hskip.5cm \text{is smooth for all} \hskip.2cm
t \in \mathopen[ 0, 1 \mathclose] ,
\\
(42) \hskip.5cm
& \Delta_\gamma ( x , 0 ) = \gamma ( h(x) ) = \gamma (x)
\hskip.5cm \text{for all} \hskip.2cm
x \in X ,
\hskip.2cm \text{by (22) and (34)},
\\
(43) \hskip.5cm
& \Delta_\gamma ( x , 1 ) = \gamma (a) = 1_G
\hskip.5cm \text{for all} \hskip.2cm
x \in X ,
\hskip.2cm \text{by (23)},
\\
(44) \hskip.5cm
& \Delta_\gamma ( a, t )
= \gamma (a) = 1_G
\hskip.5cm \text{for all} \hskip.2cm
t \in \mathopen[ 0, 1 \mathclose] ,
\hskip.2cm \text{by (33)} \hskip.2cm \text{and} \hskip.2cm (24) .
\end{aligned}
$$
With $\gamma_t = \Delta_\gamma(\cdot , t)$, this shows Claim~\ref{iiDEconnectivity+}.

\vskip.2cm

\ref{iiiDEconnectivity+}
By uniform continuity of $\Delta_\gamma$,
for all $t', t'' \in \mathopen[ 0, 1 \mathclose]$
such that $t' < t''$ and $t'' - t'$ is small enough, we have
$$
\Delta_\gamma ( x , t' )^{-1} \Delta_\gamma( x , t'') \in \mathcal V
\hskip.5cm \text{for all} \hskip.2cm
x \in X .
$$
We can therefore find a sequence $t_0, t_1, \hdots, t_k \in \mathopen[ 0, 1 \mathclose]$
such that
$$
\begin{aligned}
& 0 = t_0 < t_1 < \cdots < t_{k-1} < t_k = 1
\\
& \Delta_\gamma ( x , t_{i-1})^{-1} \Delta_\gamma( x , t_i) \in \mathcal V
\hskip.5cm \text{for all} \hskip.2cm
x \in X \hskip.2cm \text{and} \hskip.2cm i \in \{1, \hdots, k\} .
\end{aligned}
$$
Let $\log \, \colon \mathcal V \to \mathcal W$ denote the inverse
of the diffeomorphism $\exp \, \colon \mathcal W \to \mathcal V$.
For $i \in \{1, \hdots, k \}$, define
$$
\begin{aligned}
&\gamma_i (x) = \Delta_\gamma ( x , t_{i-1})^{-1} \Delta_\gamma( x , t_i)
\\
& \xi_i (x) = \log ( \gamma_i(x) ) 
\end{aligned}
$$
for all $x \in X$. Then
$$
\gamma (x) = \exp \big( \xi_1 (x) \big) \exp \big( \xi_2 (x) \big) \cdots \exp \big( \xi_k (x) \big)
\hskip.5cm \text{for all} \hskip.2cm
x \in X
$$
and $\xi_1, \hdots, \xi_k$ have the properties stated in \ref{iiiDEconnectivity+}.
\end{proof}

Whenever useful below, we consider the Lie algebra $\mathfrak g$
furnished with a scalar product,
for example that given by minus the Killing form,
so that the notion of ball makes sense in $\mathfrak g$.
\par

The next lemma shows that elements in $M (G)$
which are constant outside a small open subset of $X$
are in $M_0 (G)$.
More precisely:

\begin{lem}
\label{MandMo}
Let $a \in X$ and $g \in G$.
Let $\mathcal U$ be an open neighbourhood of $a$ in $X$.
Let $\mathcal V$ be an open neighbourhood of $g$ in $G$
and let $\mathcal W$ be an open ball centred at the origin $O$ in $\mathfrak g$
such that the map
$
\varphi \, \colon \mathcal W \to \mathcal V , \hskip.2cm
\xi \mapsto g \exp (\xi)
$
is a diffeomorphism.
Let $\delta \in M (G)$ be an element such that
$$
\begin{aligned}
&\delta (a) = g, \hskip.5cm \delta ( \mathcal U ) \, \in \, \mathcal V ,
\\
&\delta (x) \, = \, g
\hskip.2cm \text{for all} \hskip.2cm
x
\hskip.2cm \text{in some neighbourhood of} \hskip.2cm
X \smallsetminus \mathcal U .
\end{aligned}
$$
\par

Then $\delta \in M_0 (G)$.
\end{lem}

\begin{proof}
Let $\lambda \, \colon \mathopen[ 0, 1 \mathclose] \to \mathopen[ 0, 1 \mathclose]$
be a smooth function such that 
$\lambda (t) = 1$ for $t \le 1/3$ and $\lambda (t) = 0$ for $t \ge 2/3$.
Consider $\mathcal W$ with polar coordinates $(r, \sigma)$,
where ${r \in \mathopen[ 0, 1 \mathclose[}$ and $\sigma$ in the unit sphere of $\mathfrak g$.
Define a continuous map
$$
h \, \colon \mathcal W \times \mathopen[ 0, 1 \mathclose] \to \mathcal W ,
\hskip.2cm
(r, \sigma, t) \mapsto (\lambda (t) r, \sigma)
$$
for all $(r, \sigma) \in \mathcal W$ and $t \in \mathopen[ 0, 1 \mathclose]$.
Then:
$$
\begin{aligned}
&\mathcal W \to \mathcal W, \hskip.1cm \xi \mapsto h(\xi, t) ,
\hskip.2cm \text{is smooth for all} \hskip.2cm
t \in \mathopen[ 0, 1 \mathclose] ,
\\
&h(\xi, 0) = \xi
\hskip.2cm \text{for all} \hskip.2cm
\xi \in \mathcal W ,
\\
&h(\xi, 1) = O
\hskip.2cm \text{for all} \hskip.2cm
\xi \in \mathcal W ,
\\
&h(O, t) = O
\hskip.2cm \text{for all} \hskip.2cm
t \in \mathopen[ 0, 1 \mathclose] .
\end{aligned}
$$
For $t \in \mathopen[ 0, 1 \mathclose]$, define $\delta_t \, \colon X \to G$ by
$$
\begin{aligned}
&\delta_t (x) = g \exp \Big( h \left( \varphi^{-1} (\delta(x)), t \right) \Big)
\hskip.2cm \text{for all} \hskip.2cm
x \in \mathcal U ,
\\
&\delta_t (x) = g
\hskip.2cm \text{for all} \hskip.2cm
x \in X \smallsetminus \mathcal U .
\end{aligned}
$$
Then $(t \mapsto \delta_t)_{t \in \mathopen[ 0, 1 \mathclose]}$
is a continuous family in $M (G)$ of maps connecting $\delta_0 = \delta$
to the constant map $\delta_1$ of value $g$.
It follows that $\delta$ and $\delta_1$ are in the same connected component of $M (G)$.
Since $G$ is connected,
the constant map $\delta_1$ is in $M_0 (G)$,
hence $\delta \in M_0 (G)$.
\end{proof}

For a point $a \in X$, denote by $\mathcal C^\infty_{a, G}$
the group of germs at $a$ of smooth maps from $X$ to $G$.
Recall that $\mathcal C^\infty_{n, G}$ denotes
the group of germs at the origin of smooth maps from $\R^n$ to $G$.

\begin{lem}
\label{germ=germ}
Let $a \in X$.
\begin{enumerate}[label=(\roman*)]
\item\label{iDEgerm=germ}
The groups $\mathcal C^\infty_{a, G}$ and $\mathcal C^\infty_{n, G}$ are isomorphic.
\item\label{iiDEgerm=germ}
The quotient group $M_0 (G) / N_{X \smallsetminus a}$ is isomorphic
to the group of germs $\mathcal C^\infty_{a, G}$, and therefore also to $C^\infty_{n, G}$.
\end{enumerate}
\end{lem}

\begin{proof}
Claim~\ref{iDEgerm=germ} is straightforward;
indeed, any chart of $M$ around $a$ provides an isomorphism
of $C^\infty_{a, G}$ onto $C^\infty_{n, G}$.

\vskip.2cm

Consider the homomorphism
$$
\rho_a \, \colon M_0 (G) \to \mathcal C^\infty_{a, G}
$$
which associates to a globally defined smooth map $X \to G$
in $M_0 (X)$ its local germ at $a$.
The kernel of $\rho_a$ is $N_{X \smallsetminus a}$.
For Claim~\ref{iiDEgerm=germ}, it remains to show that
the homomorphism $\rho_a$ is surjective.
\par

Let $\underline \gamma \in \mathcal C^\infty_{a, G}$;
set $g = \underline \gamma (a)$.
Let $\mathcal U'$ be an open neighbourhood of $a$ in~$X$
such that there exists
a representative $\gamma_a \, \colon \mathcal U' \to G$ of $\underline \gamma$.
As in the proof of Lemma~\ref{connectivity+},
we can assume that $\mathcal U'$ is diffeomorphic
to an open ball and that $\mathcal U'$ has spherical coordinates
$(r, \sigma) \in \mathopen[ 0, 1 \mathclose[ \times S^{n-1}$;
set $\mathcal U = \{ (r, \sigma) \in \mathcal U' \mid r < 1/3 \}$.
We can also assume that there exist
$\mathcal V \subset G$ and $\mathcal W \in \mathfrak g$
as in Lemma~\ref{MandMo}, such that $\gamma (\mathcal U') \subset \mathcal V$.
\par

Let $\lambda \, \colon \mathopen [ 0, 1 \mathclose[ \to \mathopen [ 0, 1 \mathclose[$
be a smooth function such that $\lambda (t) = t$ for $t \le 1/3$
and $\lambda (t) = 0$ for $t \ge 2/3$.
Define a smooth map $h \, \colon X \to X$ by
$h(x) = h(\lambda (r), \sigma)$ for all $x = (r, \sigma) \in \mathcal U'$
and $h(x) = a$ for all $x \in X \smallsetminus \mathcal U'$.
Define $\delta \, \colon X \to G$ in $M (G)$ by
$\delta (x) = \gamma (h(x))$ for all $x \in X$.
Then
$$
\begin{aligned}
\delta (x) &= \gamma (x)
\hskip.5cm \text{for all} \hskip.2cm
x \in \mathcal U
\hskip.5cm \text{and therefore} \hskip.5cm
\underline \delta = \underline \gamma ,
\\
\delta (x) &= a
\hskip.5cm \text{for all} \hskip.2cm
x \in \mathcal U' ,
\end{aligned}
$$
and $\delta \in M_0 (G)$
by Lemma~\ref{MandMo}.
This concludes the proof of the surjectivity of~$\rho_a$.
\end{proof}

\begin{lem}
\label{Zeproblem}
Let $N$ be a maximal normal subgroup in $M_0 (G)$.
\par

There exists $a \in X$ such that $N$ contains $N_{X \smallsetminus a}$.
\end{lem}

\begin{proof}
\emph{Step one.}
Set $S = M_0 (G) / N$.
This quotient group $S$ is simple, because $N$ is maximal,
and perfect, by Lemma~\ref{perfection}.
For an open subset $\mathcal U$ of $X$,
recall that $N_\mathcal U$ has been defined as
$\{ \gamma \in M_0 (G) \mid \supp (\gamma) \subset \mathcal U \}$.
Define
$$
Y = \left\{
x \in X \hskip.2cm \bigg\vert \hskip.2cm
\begin{aligned}
& \text{there exists an open neighbourhood}
\\
& \mathcal U (x)
\hskip.2cm \text{of} \hskip.2cm
x
\hskip.2cm \text{in} \hskip.2cm
X
\hskip.2cm \text{such that} \hskip.2cm
N_{\mathcal U (x)} \subset N 
\end{aligned}
\right\} .
$$
It follows from the definition that $Y$ is open in $X$.
The purpose of Step one is to show that
$$
Y = X \smallsetminus a
\hskip.5cm \text{for some point} \hskip.5cm a \in X .
$$
\par

The inclusion $Y \subset X$ is strict.
Indeed, suppose by contradiction that $Y = X$.
For each $x \in X$, there exists an open neighbourhood $\mathcal U (x)$
such that $N_{\mathcal U (x)} \subset N$.
Let $\{x_1, \hdots, x_k\}$ be a subset of $X$ such that
$\bigcup_{i = 1}^k \mathcal U (x_i) = X$.
Let $\gamma$ be an arbitrary element in $M_0 (G)$.
By Lemma~\ref{partunity},
there exist a finite sequence of smooth maps $(\gamma_j)_{ 1 \le j \le \ell }$ in $M_0 (G)$
and a sequence of indices $(i(j))_{ 1 \le j \le \ell }$ in $\{ 1, \hdots, k \}$
such that $\gamma = \gamma_1 \cdots \gamma_\ell$,
and $\gamma_j \in N_{ \mathcal U ( x_{i(j)} ) }$
for all $j \in \{1, \hdots, \ell\}$.
This implies that $\gamma_j \in N$ for all $j \in \{1, \hdots, \ell\}$,
hence that $\gamma \in N$.
Since $\gamma$ is arbitrary, this shows that $N = M_0 (G)$,
but this is impossible because $N$ is maximal normal in $M_0 (G)$.
\par

There exists a single point $a \in X$
such that $Y = X \smallsetminus \{a\}$.
Indeed, suppose by contradiction
that there are two distinct points $a$ and $b$ in $X$ outside $Y$.
There exist disjoint neighbourhoods $\mathcal A$ of $a$ and $\mathcal B$ of $b$ in $X$;
observe that $N_\mathcal A$ and $N_\mathcal B$ commute.
Consider the canonical projection $\pi \, \colon M_0 (G) \to S$;
the definition of $Y$ implies that
$N_\mathcal A \not\subset N$, hence
$\pi (N_\mathcal A) = S$,
and similarly $\pi (N_\mathcal B) = S$.
Since $S$ is perfect, we have
$$
S = (S, S) = ( \pi (N_\mathcal A) , \pi (N_\mathcal B) )
= \pi ( N_\mathcal A , N_\mathcal B ) = \pi ( \{1_G\} ) = \{1_S\} ,
$$
and this is impossible because
$S$ is not the one element group.

\vskip.2cm

\emph{Step two.}
It remains to show that $N \supset N_{X \smallsetminus a} \hskip.2cm (= N_Y)$.
\par

Let $\gamma \in N_{X \smallsetminus a}$.
Let $\mathcal U \subset X$
and $\xi_1, \hdots, \xi_k \in \mathcal C^\infty(X, \mathfrak g)$ be as in
Lemma~\ref{connectivity+}~\ref{iiDEconnectivity+} and~\ref{iiiDEconnectivity+}.
\par

Let $x \in \supp (\xi_1)$.
Since $x \in X \smallsetminus \{a\}$,
there exists by Step one an open neighbourhood $\mathcal U (x)$ of $x$ in $X$
such that $N_{\mathcal U (x)} \subset N$.
Since $\supp (\xi_1)$ is compact,
there exists a finite subset $\{x_1, \hdots, x_\ell\}$ of $\supp (\xi_1)$
such that $\supp (\xi_1) \subset \bigcup_{1 \le j \le \ell} \mathcal U (x_j)$.
Let $\{ \lambda_0, \lambda_1, \hdots, \lambda_\ell \}$ be a smooth partition of unity
subordinated to $( X \smallsetminus \supp (\xi_1), \mathcal U (x_1), \hdots, \mathcal U (x_\ell) \}$.
For $j \in \{1, \hdots, \ell\}$,
the $\lambda_j \xi_1$~'s commute with each other,
and $\lambda_0 \xi_1 = 0$, therefore we have
$$
\exp ( \xi_1 (x) ) = \prod_{j = 1}^\ell \exp ( \lambda_j (x) \xi_1 (x) )
\hskip.5cm \text{for all} \hskip.2cm
x \in X .
$$
By definition of the $\mathcal U (x_j)$~'s,
the maps $x \mapsto \exp ( \lambda_j(x) \xi_1 (x) )$
are in $N_{\mathcal U (x_j)}$, and therefore in $N$,
hence the map $x \mapsto \exp ( \xi_1(x) )$ is in $N$.
\par

Similarly, $x \mapsto \exp ( \xi_i (x) )$ is in $N$ for all $i \in \{1, \hdots, k \}$.
It follows that their product, $\gamma$, is also in $N$,
hence that $N_{X \smallsetminus a} \subset N$.
\end{proof}

\begin{proof}[\textbf{End of proof of Proposition~I}]
Let $N$ be a maximal normal subgroup in $M_0 (G)$.
By Lemma~\ref{Zeproblem},
there exists $a \in X$ such that $N$ contains $N_{X \smallsetminus a}$.
Let $\rho_a \, \colon M_0 (G) \to \mathcal C^\infty_{n, G}$
be the epimorphism of Lemma~\ref{germ=germ}, of kernel $ N_{X \smallsetminus a}$.
Then $\rho_a (N)$ is a maximal normal subgroup of $\mathcal C^\infty_{n, G}$,
hence $\rho_a (N) = N_O \mathcal C^\infty_{n, G}$
by Proposition~II,
hence $N = N_a M_0 (G)$.
\end{proof}

\end{document}